\documentclass[12pt]{amsart}

\usepackage{graphicx}
\usepackage{amsmath}
\usepackage{amsfonts}
\usepackage{amssymb}
\usepackage{pstricks}
\usepackage{url}
\usepackage{nicefrac}
\usepackage{color}
\usepackage{subfigure}
\usepackage{esvect}
\input xy
\xyoption{all}

\newtheorem{theorem}{Theorem}

\newtheorem{definition}[theorem]{Definition}

\newtheorem{lemma}[theorem]{Lemma}

\newtheorem{proposition}[theorem]{Proposition}

\theoremstyle{definition}
\newtheorem{remark}[theorem]{Remark}

\title{Integration in \v{C}ech theories and a bound on entropy}

\author[Hern\'andez-Corbato]{Luis Hern\'andez-Corbato}
\author[Nieves-Rivera]{D. J. Nieves-Rivera}
\author[Ruiz del Portal]{Francisco R. Ruiz del Portal}
\author[S\'anchez-Gabites]{J. J. S\'anchez-Gabites}

\subjclass[2020]{55N05, 55M05, 37B40}
\keywords{}

\thanks{The authors have been supported by grant PGC2018-098321-B-I00. The second author is a beneficiary of a predoctoral contract financed by the Santander-UCM predoctoral aid program. The fourth author acknowledgments support from the Ram\'on y Cajal programme RYC2018-025843-I. On behalf of all authors, the corresponding author states that there is no conflict of interest.}

\address{Departamento de \'Algebra, Geometr\'{\i}a y Topolog\'{\i}a\\ Universidad Complutense de Madrid \\ Plaza de Ciencias 3 \\ 28040 Madrid \\ Spain}
\email{luiherna@ucm.es, davidjni@ucm.es, rrportal@ucm.es\\ jaigabites@mat.ucm.es}

\begin{document}

\begin{abstract}
The evaluation of Alexander--Spanier cochains over formal simplices in a topological space leads to a notion of integration of Alexander--Spanier cohomology classes over \v{C}ech homology classes. The integral defines an explicit and non-degenerate pairing between the Alexander--Spanier cohomology and the \v{C}ech homology. Instead of working on the limits that define both groups, most of the discussion is carried out ``at scale $\mathcal U$'', for an open covering $\mathcal U$. As an application, we generalize a result of Manning to arbitrary compact spaces $X$: we prove that the topological entropy of $f \colon X \to X$ is bounded from below by the logarithm of the spectral radius of the map induced in the first \v{C}ech cohomology group.
\end{abstract}

\maketitle

\section{Introduction}

The topological entropy \cite{adler, katokhassel1} of a dynamical system is a measure of its complexity. Positive entropy is interpreted as the existence of chaotic dynamics. The precise computation of the entropy is very complicated except for some specific examples, but often in applications it is sufficient to obtain a positive lower bound for it.

The action of a map $f \colon X \to X$ on the homology or cohomology groups of $X$ is a good indicator of how involved the dynamics is. The entropy conjecture states that for $C^1$ maps on compact manifolds the entropy is bounded from below by the logarithm of the spectral radius of the induced map $f_* \colon H_*(X) \to H_*(X)$.
The conjecture has been proved in several instances, for example in the $C^{\infty}$ case \cite{yomdin} or when $X$ is a torus \cite{entropytori} or a nilmanifold \cite{entropynilmanifolds}, but remains open in general. A classical partial result is due to Manning \cite{manning}: the conjecture holds, even for continuous maps, if only the spectral radius of $f_*$ over the first homology group is considered.

The statement in \cite{manning} holds for spaces $X$ more general manifolds, but they must still satisfy some nice local features. Manning remarks at the end of his paper: \emph{``Although \v{C}ech cohomology theory would seem most appropriate for relating cohomology eigenvalues to topological entropy as defined in \cite{adler} by refinements of open covers we have been unable to exploit this approach.''} Motivated by this, our Theorems \ref{teo:manning1} and \ref{teo:nmanning} generalize the theorem of Manning to arbitrary compact spaces in terms of \v{C}ech cohomology. When $X$ is locally connected we recover the sharp bound $h(f) \ge \mathrm{log}|\lambda|$ for every eigenvalue $\lambda$ of the map $f_*$ induced in the first \v{C}ech cohomology group while in the completely general, non--locally connected case, we obtain a weaker inequality: $h(f) \ge (\mathrm{log}|\lambda|)/d$, where $d \in \mathbb Z^+$ is the degree of $\lambda$ as an algebraic number. This bound is still positive when $|\lambda| > 1$, as Manning's original bound.

In dynamical systems \v{C}ech theories are sometimes more useful than singular cohomology and homology theories because they are better suited (due to the continuity axiom they satisfy) to describe pathological spaces with bad local topology such as strange attractors. A standard reference for these theories is \cite{Eilenberg_and_Steenrod_1952_A}. The traditional approach to \v{C}ech theories is based on simplicial complexes that are extrinsic to the space under inspection, such as the nerve of an open covering in the original description by \v{C}ech \cite{cech} or the Vietoris complex associated to it \cite{vietoris}. By contrast, in this article we work with intrinsic descriptions both for cohomology and homology. On one hand we make use of Alexander-Spanier cohomology \cite{alexandercohomology, Spanier_1948_A}, a theory whose $q$--cochains are simply functions from $(q+1)$--tuples of points of $X$ that satisfy an algebraic relation that codifies the coboundary operator. This cohomology theory coincides with \v{C}ech cohomology theory in paracompact spaces \cite{hurewiczdugundjidowker} but does not have a simple dual construction of chains and homology groups. Regarding homology, we take the approach of Vietoris. For a fixed open covering $\mathcal U$ of $X$ we work with $\mathcal U$--small chains and define a homology group at scale $\mathcal U$ that disregards the local structure of $X$ contained in any element of $\mathcal U$ and is a reinterpretation of (and, in particular, isomorphic to) the homology group of the Vietoris complex $V(\mathcal U)$. The limit of these groups as $\mathcal U$ ranges over the open coverings of $X$ is called the Vietoris homology group of $X$ (the original formulation considered open coverings by balls of radius $\epsilon$ in a metric space \cite{hockingyoung}). In paracompact spaces it is isomorphic to the \v{C}ech homology group \cite{lefschetz, hurewiczdugundjidowker}.

One of the main ingredients in the generalization of Manning's theorem is a bilinear pairing between \v{C}ech homology and cohomology that is interesting on its own. In Section \ref{sec:integral} we define a notion of integration of Alexander--Spanier cohomology classes over \v{C}ech homology classes reminiscent of the integration of differential forms over simplices. The definition is possible thanks to the intrinsic descriptions mentioned above and is based on the evaluation of the Alexander-Spanier cochains over sufficiently small simplices, which we call ``formal'' because their information is reduced to its set of vertices. Integration defines an explicit and non-degenerate bilinear pairing between \v{C}ech cohomology and homology on compact topological spaces in a similar fashion to what integration of differential forms does between de Rham's cohomology and singular homology on differentiable manifolds. It is worth noticing that the concept of formal simplex that we use was already present in the literature from the early developments of algebraic topology under the name of abstract or symbolic simplex, see for example \cite{alexandercohomology}. Also, it was known that Alexander-Spanier cohomology and Vietoris homology theories could be constructed from the same system of complexes \cite{hurewiczdugundjidowker}. Our contribution in this framework is to establish an elementary and explicit pairing between these theories. As a side note, there are related notions of integration in the literature, see \cite{asada} or \cite[p. 128]{halperin}.

%integration of Alexander--Spanier cochains over singular simplices was discussed in \cite{asada}. \textcolor{red}{CITAR AL JAPO O NO CITAR}

Alexander-Spanier cocycles can be stratified in terms of the open covering of $X$ where the cocycle condition holds. In view of the definition of the integral, an Alexander-Spanier cocycle whose coboundary identically vanishes at each element of $\mathcal U$ (which we gather in a group denoted $H^*_{\mathcal U}(X)$) can be integrated over $\mathcal U$--small chains. In Section \ref{sec:nondegenerate} we prove that the integral defines a nondegenerate pairing at scale $\mathcal U$ if coefficients are taken in a field. The article studies these cohomology and homology groups at scale $\mathcal U$ for a fixed $\mathcal U$. They lead to straightforward proofs of the results in \cite{finite_cech}, where \v{C}ech homology was obtained as a limit of homology groups based on $\epsilon$--continuous simplices (see Section \ref{sec:homology}), and give a clear interpretation of a result of Keesee \cite{keesee} that states that any Alexander-Spanier cohomology class has a representative that takes a finite number of different values (see Proposition \ref{prop:imagenfinita}).

Sections \ref{sec:background} and \ref{sec:homology} contain background material on \v{C}ech theories, notably the definition of Alexander-Spanier cohomology and the homology groups at scale $\mathcal U$. Some results about the first homology group and the cohomology groups are presented in Sections \ref{sec:1dim} and \ref{sec:imagenfinita}. The integral is defined in Section \ref{sec:integral}, and we prove its nondegeneracy in the ensuing section. The results on entropy and their proofs build on all the preliminary work and are presented in the last part of the article. We also provide examples that show the need to use eigenvalues in cohomology instead of homology to bound the entropy. The non--locally connected case requires some technical work that is the content of Section \ref{sec:nmanning}.

\section{Background on \v{C}ech homology and cohomology}\label{sec:background}

\subsection{Classical definition}

Let $X$ be any topological space. Consider an open covering $\mathcal{U}$ of $X$. There are two classical simplicial complexes associated to $\mathcal{U}$. One is its nerve (or \v{C}ech complex \cite{cech}) $N(\mathcal U)$, which is defined as follows:

\begin{itemize}
	\item[(i)] It has a vertex for each member of $\mathcal{U}$.
	\item[(ii)] A (finite) collection of vertices spans a simplex if and only if the corresponding members of $\mathcal{U}$ have a nonempty intersection.
\end{itemize}

Another one, which will be more useful for us, is the Vietoris complex $V(\mathcal U)$ of $\mathcal{U}$. Its original conception \cite{vietoris} and later and recent developments (which have shifted their denomination to Vietoris-Rips complex) have been mainly restricted to metric spaces but here we present a definition in terms of an open covering:
\begin{itemize}
	\item[(i)] Its vertices are the points of $X$.
	\item[(ii)] A collection of vertices $x_0,\ldots,x_q$ spans a simplex in $V(\mathcal{U})$ if there exists $U \in \mathcal{U}$ that contains all of them.
\end{itemize}

These two complexes are, in the sense explained by Dowker \cite{dowker1}, ``dual'' to each other and, in particular, have isomorphic homology and cohomology groups.

Refining the open covering $\mathcal{U}$ produces complexes that approximate the space $X$ ever better, and in the limit the homology and cohomology groups of $N(\mathcal{U})$ or $V(\mathcal{U})$ constitute algebraic topological invariants of $X$. Let us describe this more precisely. Given a refinement $\mathcal V$ of $\mathcal U$, there are simplicial maps $N(\mathcal V) \to N(\mathcal U)$ and $V(\mathcal V) \to V(\mathcal U)$. The latter is an inclusion while the former is not uniquely defined, because $V \in \mathcal V$ may be contained in several $U \in \mathcal U$, but all the choices give rise to contiguous maps and, in particular, are homotopic \cite[Section 3.5]{spanier1}.
These bonding maps provide the family of all (\v{C}ech or Vietoris) complexes that arise from open coverings $\mathcal U$ of $X$ the structure of an inverse system, denoted by $\{N(\mathcal U)\}$ and $\{V(\mathcal U)\}$, respectively. Applying the homology functor we obtain two inverse systems of homology groups. Their inverse (or projective) limits, referred to as \v{C}ech and Vietoris homology groups, are evidently isomorphic by the result of Dowker, although this fact was already known by that time \cite[VII (26.1)]{lefschetz}. The construction for cohomology is analogous, but let us remark that there is no Vietoris cohomology theory formulated as such. Since the cohomology functor is contravariant we obtain directed systems of groups that are pro--isomorphic and, in particular, have isomorphic (direct) limits. The groups defined in this way are called \v{C}ech homology and \v{C}ech cohomology groups and denoted by $\check{H}_*(X)$ and $\check{H}^*(X)$, respectively.

The coefficient group becomes relevant in some parts of the article. Unless explicitly stated, we assume it is an $R$--module $G$ and usually omit it from the notation.

\subsection{Alexander--Spanier cohomology}

As mentioned in the Introduction, it is also possible to describe the cohomology groups $\check{H}^q(X)$ in an intrinsic manner using the language of Alexander--Spanier cohomology. We first recall its definition from \cite{spanier1} (see also \cite{Spanier_1948_A}). For $q \geq 0$ a $q$--cochain is a (not necessarily continuous) map $\xi \colon X^{q+1} \to G$. The set $C^q(X)$ of $q$--cochains forms an $R$--module. The coboundary of a cochain is defined in the usual fashion: $\delta \xi(x_0,\ldots,x_{q+1}) := \sum_{j=0}^q (-1)^j \xi(x_0,\ldots,\widehat{x_j},\ldots,x_q)$. A cochain is said to be locally zero if there exists an open covering $\mathcal{U}$ of $X$ such that $\xi(x_0,\ldots,x_q) = 0$ whenever $x_0,\ldots,x_q$ all belong to the same element of $\mathcal{U}$. The subset of locally zero cochains is a submodule of $C^q(X)$ and we denote by $\overline{C}^q(X)$ the quotient module of all cochains modulo the locally zero ones. The coboundary of a locally zero cochain is again locally zero, and so $\delta$ descends to a coboundary operator on $\overline{C}^q(X)$ which we denote again with the same symbol. The cohomology of the complex $\{\overline{C}^q(X),\delta\}$ is, by definition, the Alexander--Spanier cohomology of $X$ (with coefficients in $G$). It is isomorphic to the \v{C}ech cohomology of $X$ when $X$ is paracompact (\cite{hurewiczdugundjidowker} and \cite[Corollary 8, p. 334]{spanier1}) or, in general, when they are defined using the same family of coverings \cite{dowker1}. 

Now suppose an open covering $\mathcal{U}$ of $X$ is fixed. $\mathcal U$ shall be thought of as the ``scale'' below which the structure of $X$ is disregarded. With minor modifications (already mentioned by Spanier \cite[Appendix A]{Spanier_1948_A}) the above definitions can be adapted to define another cohomology group as follows. Say that a $q$--cochain $\xi$ is $\mathcal{U}$--locally zero if $\xi(x_0, \ldots, x_q) = 0$ whenever $x_0,\ldots,x_q$ belong to some member of $\mathcal{U}$. Heuristically, these cochains ignore all the structure of $X$ below scale $\mathcal{U}$. It is straightforward to check that the coboundary of a $\mathcal{U}$--locally zero cochain is again $\mathcal{U}$--locally zero. Thus we may quotient $C^q(X)$ by the submodule of $\mathcal{U}$--locally zero cochains to obtain a module $\overline{C}^q_{\mathcal{U}}(X)$. The coboundary operator descends to a coboundary operator on this module yielding a cochain complex $\{\overline{C}^q_{\mathcal{U}}(X),\delta\}$. We denote the cohomology of this complex by $H^q_{\mathcal{U}}(X)$. This is, therefore, very similar to the definition of the Alexander--Spanier cohomology of $X$ but using a single covering $\mathcal{U}$ instead of letting $\mathcal{U}$ run over all open coverings. Notice that, by definition, the class $\overline{\xi} \in \overline{C}^q_{\mathcal{U}}(X)$ of a cochain $\xi$ is a cocycle if and only if $\delta \xi$ vanishes over every member of $\mathcal{U}$; similarly, $\overline{\xi} = \overline{\eta}$ means that $\xi$ and $\eta$ agree when evaluated on tuples that are contained in some member of $\mathcal{U}$.

If $\mathcal{V}$ refines $\mathcal{U}$ then every $\mathcal{U}$--locally zero cochain is also $\mathcal{V}$--locally zero and so there is a natural homomorphism $\overline{C}^q_{\mathcal{U}}(X) \to \overline{C}^q_{\mathcal{V}}(X)$. This map commutes with the coboundary operator and hence induces a homomorphism $\pi_{\mathcal U \mathcal V} \colon H^q_{\mathcal{U}}(X) \to H^q_{\mathcal{V}}(X)$. There is, then, a direct system whose terms are $H^q_{\mathcal{U}}(X)$ and whose bonding maps are the maps $\pi_{\mathcal U \mathcal V}$.

\begin{proposition}\label{prop:limUcohomologia} The direct limit of $\{ H^q_{\mathcal{U}}(X) ; \pi_{\mathcal U \mathcal V}\}$ is precisely the Alexander--Spanier cohomology of $X$. Therefore, under general hypothesis (described above) it is also isomorphic to the \v{C}ech cohomology of $X$, $\check{H}^q(X)$.
\end{proposition}
\begin{proof} It follows directly from its definition that $\overline{C}^q(X)$ can be identified with the direct limit of $\{ \overline{C}^q_{\mathcal{U}}(X)\}$. Then the result owes to the fact that the homology functor commutes with direct limits.
\end{proof}

\section{An alternative description of \v{C}ech homology}\label{sec:homology}

In this section we perform a quick reformulation of the Vietoris homology theory. We introduce the notion of a formal simplex small with respect to an open covering $\mathcal U$ and construct the homology groups at scale $\mathcal U$, $H_*^{\mathcal U}(X)$. These are isomorphic to $H_*(V(\mathcal U))$. As an application we recover easily the results of \cite{finite_cech} on $\epsilon$--homology groups.

A formal $q$--simplex in $X$ is just an ordered collection $\sigma = (x_0\ \ldots\ x_q)$ of $(q+1)$ points in $X$. We call the $x_i$ the vertices of $\sigma$ and say that $\sigma$ is $\mathcal U$--small if all its vertices are contained in some element of $\mathcal U$. (A formal $q$--simplex is just a point in $X^{q+1}$, but it is best to picture it as a subset of $X$).

A formal $q$--chain in $X$ is a finite linear combination $c = \sum_i k_i \sigma_i$ of formal $q$--simplices $\sigma_i$ with coefficients $k_i \in G$. We shall say that $c$ is $\mathcal{U}$--small if all the $\sigma_i$ are $\mathcal{U}$--small and denote by $S_q^{\mathcal{U}}(X)$ the set of all $\mathcal{U}$--small formal $q$--chains. 
The boundary of a formal simplex $\sigma = (x_0 \ \ldots \ x_q)$ is defined in the usual way to be the formal chain $\partial \sigma = \sum_{j=0}^q (-1)^j (x_0 \ \ldots\ \widehat{x_j}\ \ldots \ x_q)$, where $(x_0\ \ldots\ \widehat{x_j}\ \ldots \ x_q)$ is the $(q-1)$--formal simplex determined by the ordered collection $x_0,\ldots,x_q$ from which $x_j$ has been removed. The boundary of a $\mathcal{U}$--small formal simplex is clearly $\mathcal{U}$--small again, and so there is a chain complex $\{S_q^{\mathcal{U}}(X),\partial\}$. We shall denote its homology by $H_q^{\mathcal{U}}(X;G)$ and refer to it as the homology of $X$ at scale $\mathcal{U}$.
%Clearly the chain complex $\{S_q^{\mathcal{U}}(X),\partial\}$ is isomorphic to the simplicial chain complex of $V(\mathcal{U})$ and so in particular $H_q^{\mathcal{U}}(X) = H_q(V(\mathcal{U}))$.

Note that if we disregard the ordering of the vertices that define a formal simplex and discard simplices with repeated vertices we obtain the geometric simplices from the Vietoris complex $V(\mathcal U)$. Otherwise stated, formal simplices are just ordered simplices of $V(\mathcal U)$. Hence there is a map from $S_q^{\mathcal{U}}(X)$ to the simplicial chains of $V(\mathcal{U})$ defined by sending a formal simplex $(x_0 \ \ldots \ x_q)$ to the same oriented simplex of $V(\mathcal{U})$ if all the vertices are different and to zero otherwise. This map induces isomorphisms in homology (see for example \cite[Theorem 13.6, p. 77]{munkres}) and therefore $H_q(V(\mathcal U);G)$ can be identified with $H_q^{\mathcal{U}}(X;G)$. Furthermore, if $\mathcal V$ refines $\mathcal U$, the previous isomorphism conjugates the bonding maps $H_q(V(\mathcal V)) \to H_q(V(\mathcal U))$ and $H^{\mathcal V}_q(X) \to H^{\mathcal U}_q(X)$. Therefore, the limit of the inverse system of groups $\{H^{\mathcal U}_q(X)\}$ is isomorphic to the Vietoris $q$--th homology group, which is in turn isomorphic to the \v{C}ech homology group $\check{H}_q(X)$. Throughout this article we directly refer to the limit of the homology groups at scale $\mathcal U$ as the \v{C}ech homology group.

An element $\gamma \in \check{H}_q(X)$ consists of a family $(\gamma^{\mathcal{U}})_{\mathcal{U}}$ of homology classes, where each $\gamma^{\mathcal{U}} \in H^{\mathcal{U}}_q(X)$. In turn each $\gamma^{\mathcal{U}}$ is represented by a (nonunique) chain $c$ that is $\mathcal{U}$--small. Somewhat abusing language, we shall call such a chain $c$ a \emph{$\mathcal{U}$--small representative of $\gamma$}.

Intuitively $H_q^{\mathcal{U}}(X)$ ignores all the structure of $X$ at scales smaller than $\mathcal{U}$. This idea is expressed in the following mathematical statement:

\begin{proposition} \label{prop:Ucycle} Let $c = \sum_{i=0}^n k_i \sigma_i$ be a $q$--cycle with $q \geq 1$. Assume that the vertices of all the $\sigma_i$ are contained in a single element $U \in \mathcal{U}$. Then $c$ is nullhomologous in $H_q^{\mathcal{U}}(X)$.
\end{proposition}
\begin{proof}
Pick an arbitrary point $a \in U$. For every formal simplex $\sigma = (x_0 \ x_1 \dots \, x_p)$ whose vertices lie in $U$ consider $C_a(\sigma) = (a \ x_0 \ x_1 \ldots \, x_p)$, a sort of formal cone of $\sigma$ over $a$ which is still contained in $U$. This operation on formal simplices satisfies

\centerline{
$\partial C_a(\sigma) + C_a(\partial \sigma) = \sigma$
}

\noindent and the identity holds for chains as well if we extend $C_a$ linearly. If we apply the identity to a cycle $c$, we obtain that $c = \partial C_a(c)$ is a boundary and the conclusion follows.

\end{proof}

For later reference we record here the following remark which is not obvious from the definition of $H_*^{\mathcal{U}}(X)$:

\begin{remark} \label{rem:f_gen} If the covering $\mathcal{U}$ is finite, then $H^{\mathcal{U}}_*(X)$ is finitely generated.
\end{remark}
\begin{proof} Since $\mathcal{U}$ is finite its nerve $N(\mathcal{U})$ is a finite simplicial complex and therefore $H_*(N(\mathcal{U}))$ is finitely generated too. By the result of Dowker \cite{dowker1} mentioned earlier this is isomorphic to $H_*(V(\mathcal{U}))$, which in turn is isomorphic to $H^{\mathcal{U}}_*(X)$.
\end{proof}

Compared with singular simplices, formal simplices carry very little information: just the vertices of the simplex. For this reason it is very easy to define functors from other homology theories to the homology groups $H^{\mathcal{U}}_*(X)$. To illustrate this idea we consider the description of \v{C}ech homology in terms of $\epsilon$--continuous simplices given in \cite{finite_cech}. First, we need some definitions. The space $X$ will be assumed to be compact and metric. A map $f$ between metric spaces is called $\epsilon$--continuous if there exists $\delta > 0$ such that $d(f(x), f(x')) < \epsilon$ for every $x, x'$ at a distance at most $\delta$. Roughly speaking, $f$ does not exhibit discontinuities of size greater than $\epsilon$.

Let $\Delta_q$ denote the standard $q$--simplex. An $\epsilon$--continuous $q$--simplex is an $\epsilon$--continuous map $\sigma : \Delta_q \to X$. The free group generated by these is denoted by $S^{\epsilon}_q(X)$. There is a boundary operator $\partial : S_q^{\epsilon}(X) \to S_{q-1}^{\epsilon}(X)$ defined in the exact same way as in the singular theory and satisfies $\partial^2 = 0$.  The homology of the chain complex $(S^{\epsilon}_*(X), \partial)$ is denoted by $H^{\epsilon}_*(X)$. Given $0 < \epsilon' < \epsilon$, any $\epsilon'$--continuous simplex is automatically an $\epsilon$--continuous simplex, so there exist natural inclusion maps $S^{\epsilon'}_q(X) \to S^{\epsilon}_q(X)$ which carry over to the homology groups and yield an inverse system $\{H^{\epsilon}_*(X)\}$. Then:
\begin{theorem}[{\cite[Corollary 1, p. 88]{finite_cech}}]
The inverse limit of the system is isomorphic to $\check{H}_*(X)$.
\end{theorem}
The proof in \cite{finite_cech} proceeds by establishing that the homology theory just defined has the continuity property and agrees with the singular theory over simplicial complexes, and therefore must coincide with \v{C}ech homology. We shall give an alternative proof by establishing an isomorphism between the limits of $\{H^{\epsilon}_*(X)\}$ and $\{H^{\mathcal U}_*(X)\}$.

The main ideas are very natural: given a (small in a sense later explained) $\epsilon$--continuous simplex $\sigma$ we may produce a formal simplex $D(\sigma)$ by disregarding all the information in $\sigma$ except for its vertices; conversely, given a formal simplex $\tau$ we may ``interpolate'' to produce an $\epsilon$--continuous simplex $I(\tau)$ whose vertices are those of the formal simplex.% 

We begin by replacing the chain complex $S^{\epsilon}_q(X)$ with its subcomplex $E^{\epsilon}_q(X)$ generated by the ($\epsilon$--continuous) maps $\tau \colon \Delta_q \to X$ whose image has a diameter less than $\epsilon$. Note that if $\tau \in S^{\epsilon}_q(X)$ is an arbitrary $\epsilon$--continuous simplex there exists $\delta$ such that $\tau$ carries any set of diameter less than $\delta$ to a set of diameter less than $\epsilon$ and so subdividing $\Delta_q$ barycentrically enough times we see that the restriction of $\tau$ to each member of the subdivision belongs to $E_q^{\epsilon}(X)$. Then a standard argument (see for example \cite[\S 31, pp. 175ff.]{munkres} or \cite[Appendix I, p. 211ff.]{vick1}) shows that the inclusion $i \colon E_q^{\epsilon}(X) \subseteq S^{\epsilon}_q(X)$ induces isomorphisms between the homology of both chain complexes.

\smallskip

{\it Discretization.} We now define a ``discretization operator'' $D \colon E^{\epsilon}_q(X) \to S^{\mathcal{U}}_q(X)$. Fix an open covering $\mathcal{U}$ of $X$ and let $\epsilon > 0$ be a Lebesgue number for $\mathcal{U}$. Recall that every chain in $E^{\epsilon}_q(X)$ is of the form $\sum k_i \tau_i$ where each $\tau_i \colon \Delta_q \to X$ is a map whose image has a diameter less than $\epsilon$ and is therefore contained in some member of $\mathcal{U}$. Hence we can define $D(\tau_i)$ to be the $\mathcal{U}$--small formal simplex $(\tau(v_0) \ \ldots \ \tau(v_q))$ and then extend linearly to chains by $D(c) = \sum k_i D(\tau_i)$. Geometrically, $D$ simply forgets all the information that a chain carries except for the vertices of its constituent summands. It is straightforward to check that $D$ commutes with the boundary operator so it induces a homomorphism in homology.

\smallskip

{\it Interpolation.} Now we want to define an ``interpolation operator'' $I \colon S^{\mathcal{U}}_q(X) \to E^{\epsilon}_q(X)$ which essentially ``fills in'' formal simplices to turn them into maps defined on all of $\Delta_q$. Although this can be done in several ways, it should be compatible with the boundary operator. Fix $\epsilon > 0$ and choose a covering $\mathcal{U}$ of $X$ whose members all have a diameter less than $\epsilon$.

We introduce the auxiliary map $J_q \colon \Delta_q \to \Delta_q$ defined as follows. Given a point $p \in \Delta_q$ with barycentric coordinates $(\lambda_0,\ldots,\lambda_q)$, find the smallest $i$ such that $\lambda_i > 0$ and set $J_q(p) := v_i$, the $i$-th vertex of $\Delta_q$. Given a $\mathcal{U}$--small formal simplex $(x_0 \  \ldots \ x_q)$ in $X$, we define a piecewise constant map $\tau \colon \Delta_q \to X$ at the vertices by $\tau(v_i) = x_i$ and then set $\tau(p) = \tau(J_q(p))$ for any other $p \in \Delta_q$.

%\begin{figure}[h!]
%\begin{pspicture}(0,0)(14,7.5)
%\psgrid(0,0)(14,7.5)
%\rput[bl](1,0){\scalebox{0.7}{\includegraphics{discretize.eps}}}
%\rput(0.8,4){$v_0$} \rput(3.5,4){$v_1$} \rput(0.9,6.8){$v_2$} 
%\rput(1.8,5){$\Delta_2$}
%\rput(6.5,0.6){$X$} \rput(8.6,1.3){$x_0$} \rput(9.45,1.9){$x_1$} 
%\rput(10.6,1.4){$x_2$}
%\rput(4,5.4){$=$} \rput(8,5.4){$\cup$} \rput(11.6,5.4){$\cup$}
%\end{pspicture}
%\caption{From a formal simplex to a (discontinuous) simplex \label{fig:interpolate}}
%\end{figure}

The image of $\tau$ is precisely $\{x_0,\ldots,x_q\}$, which is contained in some member of $\mathcal{U}$. In particular the diameter of the image of $\tau$ is less than $\epsilon$, and so $\tau$ is indeed an element of $E^{\epsilon}_q(X)$. We can therefore define a homomorphism $I \colon S^{\mathcal{U}}_q(X) \to E^{\epsilon}_q(X)$ by linearly extending the definition above to chains. It is straightforward to check that $I$ commutes with the boundary operators so it induces a homomorphism in homology.

\medskip

Let us examine the composition $D \circ I$. First we need to find what its source and target spaces are: given an open covering $\mathcal{U}$ we find a Lebesgue number $\epsilon$ for $\mathcal{U}$ and then another covering $\mathcal{U}'$ whose members all have a diameter less than $\epsilon$. Then $D \circ I \colon S^{\mathcal{U}'}_q(X) \to S^{\mathcal{U}}_q(X)$. It is straightforward to check that the map $D \circ I$ is just the inclusion $S^{\mathcal{U}'}_q(X) \subseteq S^{\mathcal{U}}_q(X)$.

The analysis of the reverse composition $I \circ D$ is only slightly more complicated. Given $\epsilon$, we find a covering $\mathcal{U}$ whose members have a diameter less than $\epsilon$ and then in turn some $\epsilon' < \epsilon$ which is a Lebesgue number for $\mathcal{U}$. Then the $I \circ D$ is a map $E^{\epsilon'}_q(X) \to E^{\epsilon}_q(X)$ which works in the following way: given a simplex $\tau' \colon \Delta_q \to X$, it discretizes it by retaining only its values on the vertices of $\Delta_q$ and then interpolates it to produce a new simplex $\tau \colon \Delta_q \to X$ using the prescription given above for $I$. %Notice that in general $\tau$ and $\tau'$ are different; for example, the former always has a finite image whereas the latter need not. Also, the image of $\tau = I \circ D (\tau')$ is a subset of the image of the original $\tau$ and so has a diameter less than $\epsilon'$. Hence, although formally the composition $I \circ D$ carries $E^{\epsilon'}_q(X)$ to $E^{\epsilon}_q(X)$, in fact its image lies in $E^{\epsilon'}_q(X)$ again.

The simplices $\tau$ and $\tau'$ can be easily connected through an $\epsilon$-continuous homotopy $\{\tau^t\} \colon \Delta_q \times [0,1] \to X$ defined by $\tau^0 = \tau'$ and $\tau^t = \tau$ for $0 < t \le 1$. This construction allows to define a prism operator $P \colon E^{\epsilon'}_q(X) \to E^{\epsilon}_{q+1}(X)$ (see for instance \cite[Theorem 2.10, p. 112]{hatcher1}) that satisfies

\centerline{
$P (\partial \tau') + \partial \,P(\tau') = \tau - \tau'$
}
\noindent Then, $P$ provides a chain equivalence between $I \circ D$ and the inclusion $E^{\epsilon'}_q(X) \to E^{\epsilon}_q(X)$.

%Since both $I$ and $D$ commute with the appropriate boundary operators, they induce well defined homomorphisms in homology. Summing up all the discussion so far we then have:
%\begin{itemize}
%	\item[(i)] Given $\epsilon$ and a covering $\mathcal{U}$ whose members have a diameter bounded above by $\epsilon$, a homomorphism \[I_* \colon H^{\mathcal{U}}_*(X) \to H^{\epsilon}_*(X).\]
%	\item[(ii)] Given an open covering $\mathcal{U}$ and a Lebesgue number $\epsilon$ for $\mathcal{U}$, a homomorphism \[D_* \colon H^{\epsilon}_*(X) \to H^{\mathcal{U}}_*(X).\]
%	\item[(iii)] These homomorphisms have the property that $D_* \circ I_*$ and $I_* \circ D_*$ coincide with the bonding maps of the inverse systems $\{H^{\mathcal{U}}_*(X)\}$ and $\{H^{\epsilon}_*(X)\}$ respectively.
%\end{itemize}

To sum up, the composite maps  $D_* \circ I_*$ and $I_* \circ D_*$ coincide with the bonding maps of the inverse systems $\{H^{\mathcal{U}}_*(X)\}$ and $\{H^{\epsilon}_*(X)\}$ respectively.

Now we can easily show that the limits of the systems $\{H^{\epsilon}_*(X)\}$ and $\{H^{\mathcal{U}}_*(X)\}$ are isomorphic. For if $\mathcal U(\epsilon)$ is the covering of $X$ given by all the open balls of radius $\epsilon$ and $\epsilon'$ is very small, the bonding map $H^{\epsilon'}_*(X) \to H^{\epsilon}_*(X)$ is equal to $I_* \circ D_*$ and so it factors through $H^{\mathcal U(\epsilon)}_*(X)$. Similarly, if $\epsilon$ is a Lebesgue number for an open covering $\mathcal U$ of $X$ and all the elements of a refinement $\mathcal U'$ of $\mathcal U$ have diameter smaller than $\epsilon$, the bonding map $H^{\mathcal U'}_*(X) \rightarrow H^{\mathcal U}_*(X)$ factors through $H^{\epsilon}_*(X)$. Therefore, inserting alternatively homology groups at scale $\mathcal U$ and $\epsilon$-homology groups in an appropriate fashion, we build an inverse sequence that can be interpreted as a cofinal subsystem both of $\{H^{\epsilon}_*(X)\}$ and $\{H^{\mathcal{U}}_*(X)\}$. Since the inverse limit of an inverse system and any cofinal system are isomorphic, it follows that $\varprojlim_{\epsilon} H^{\epsilon}_*(X)$ is isomorphic to $\varprojlim_{\mathcal{U}} H^{\mathcal{U}}_*(X)$, which is in turn isomorphic to $\check{H}_*(X)$.

\section{$1$--dimensional homology}\label{sec:1dim}

We devote this section to an analysis of the groups $H_1^{\mathcal{U}}(X;G)$ and $\check{H}_1(X;G)$. We will not be systematic at all, but only pursue the subject to the extent that is needed for the proofs of the results about entropy in Sections \ref{sec:manning} and \ref{sec:nmanning}.% It will also illustrate how to handle $\mathcal U$--small chains and ultimately work with \v{C}ech homology. 

\medskip

{\bf Terminology.} In the sequel we will mostly work with the homology groups $H^{\mathcal{U}}_*(X)$ and therefore with $\mathcal{U}$--small formal simplices and chains. For the sake of brevity we will drop the expression ``formal'' to refer to the simplices, chains, etc.
\medskip

First we show that $H_1^{\mathcal{U}}(X;\mathbb Z)$ is generated by special cycles, which we call \emph{simple} and \emph{elementary}, of the form

\centerline{
$(x_1 \ x_2) + (x_2 \ x_3) + \ldots + (x_{n-1} \ x_{n})$
}

\noindent where $x_{n} = x_1$ (cycle condition), all the $x_1,\ldots,x_{n-1}$ are different points in $X$ (\emph{elementary}) and it is possible to choose pairwise distinct $U_i \in \mathcal U$ so that each $\mathcal U$--small simplex $(x_i \ x_{i+1})$ is contained in $U_i$ (\emph{simple}). Intuitively, a cycle is simple if it has no ``self-intersections'' at scale $\mathcal{U}$. A simple elementary chain has an analogous definition dropping the condition that $x_1 = x_n$ and requiring that all the $x_i$ be different. Notice that an elementary chain, not necessarily simple, corresponds to an edge path in $V(\mathcal U)$ and if the chain is closed (i.e. it is a cycle) so is the path. If the chain or cycle is also simple then the edge path can be projected to an equivalent edge path $N(\mathcal U)$.

Although the ensuing discussion is very elementary in nature and the details could have been left to the reader, we include it to illustrate how to manipulate $\mathcal U$--small cycles.
The first step towards the proposed description is the following lemma:
\begin{lemma} \label{lem:reduction1} Let $c$ be a $\mathcal{U}$--small cycle. Then $c \sim \sum_i k_i e_i$ where all the $e_i$ are $\mathcal{U}$--small elementary cycles and $k_i \in G$. Moreover, every simplex that appears in some $e_i$ was already part of $c$.
\end{lemma}

There are several alternative ways to prove the assertion. Let $\{x_i\}$ be the (finite) collection of vertices among the $\mathcal U$--small 1--simplices that constitute $c$. Then, $c$ can be seen as a (homology) cycle in the complete graph $K$ with vertices $\{x_i\}$. Since the first homology group of a finite graph is generated by closed paths and closed paths in $K$ correspond to elementary cycles in $H_1^{\mathcal{U}}(X;\mathbb Z)$, $c$ is an integer combination of elementary cycles as desired.

%  (recall that if we start with an spanning tree/forest of a finite graph, a set of generators of its first homology group can be constructed as we add the remaining edges one by one to the growing subgraphs: for every edge take a closed path composed of the selected edge and the path in the subgraph that connects its vertices)

It only remains to write an elementary cycle as a combination of simple elementary cycles. With a view towards later applications we are going to describe in detail the process when applied to an arbitrary elementary $\mathcal{U}$--small chain $e = \sum_{i=1}^{n-1} (x_i \ x_{i+1})$, not necessarily a cycle. 

\begin{lemma} \label{lem:red2} Let $e$ be a $\mathcal{U}$--small elementary chain. Then $e \sim e_0 + s_1 + \ldots + s_r$ where $e_0$ is a simple elementary chain and the $s_i$ are simple elementary cycles.
\end{lemma}
\begin{proof} As usual, take $e = \sum_{i=1}^{n-1} (x_i \ x_{i+1})$. For each $i$ choose an arbitrary $U_i \in \mathcal{U}$ that contains $(x_i \ x_{i+1})$, and suppose that $U_i = U_j$ for some $j > i$ (if this does not happen, then $e$ is already simple). Fix any such $i$ and choose $j > i$ to be the smallest with the property $U_i = U_j$. Consider the portion (*) of the chain between indices $i$ and $j$: \[e = \ldots + (x_{i-1} \ x_i) + \underbrace{(x_i \ x_{i+1}) + \ldots + (x_j \ x_{j+1})}_{(*)} + (x_{j+1} \ x_{j+2}) + \ldots\] We remove the simplices $(x_i \ x_{i+1})$ and $(x_j \ x_{j+1})$ from $e$ and replace them with $(x_i \ x_{j+1}) + (x_j \ x_{i+1})$ to obtain a new chain $e'$, still $\mathcal{U}$--small. Moreover, \[e - e' = (x_i \ x_{i+1}) + (x_j \ x_{j+1}) - (x_i \ x_{j+1}) - (x_j \ x_{i+1})\] is easily checked to be a cycle and it is entirely contained in $U_i = U_j$, so by Proposition \ref{prop:Ucycle} it is nullhomologous. Hence $e \sim e'$. Figures \ref{fig:pathe} and \ref{fig:pathep} provide a schematic drawing of the situation. The dotted segments connecting the vertices $x_k$ are just pictorial aids to suggest the simplices of the chains, but of course they are not true subsets of $X$. 

\begin{figure}[h]
\begin{minipage}{0.48\textwidth}
\includegraphics[scale = 0.95]{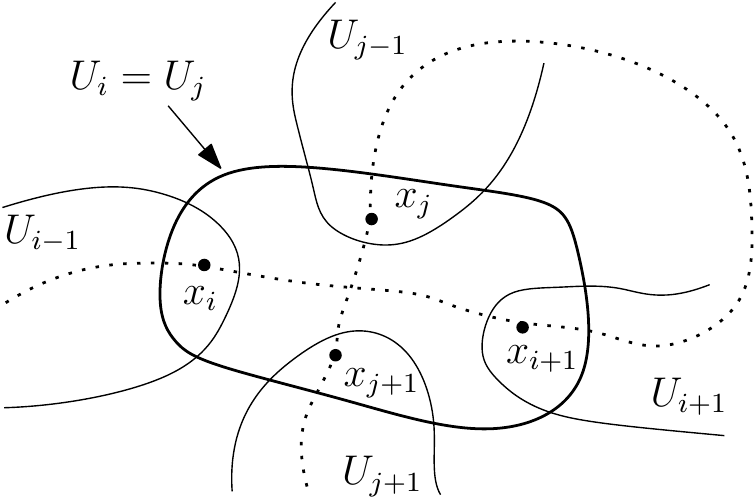}
\caption{Original chain $e$ \label{fig:pathe}}
\end{minipage}
\hfill
\begin{minipage}{0.48\textwidth}
\includegraphics[scale = 0.95]{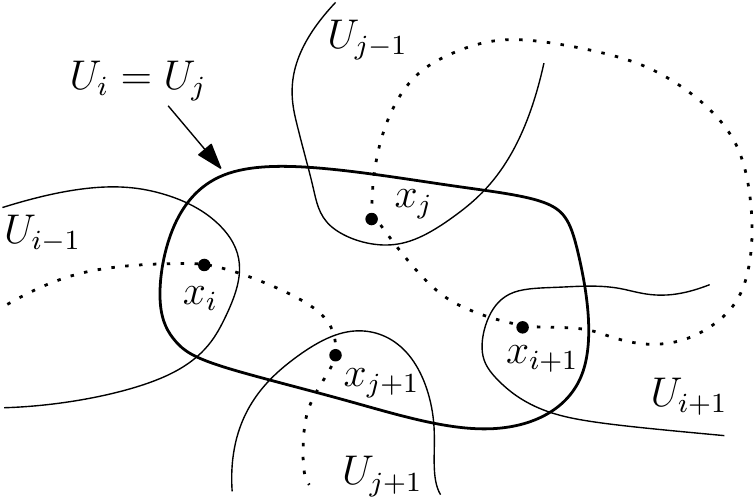}
\caption{Modified chain $e'$ \label{fig:pathep}}
\end{minipage}
\end{figure}

%\begin{figure}[h]

%\begin{pspicture}(0,0)(13,7.5)
%\psgrid(0,0)(13,7.5)
%\rput[bl](0,0.5){\scalebox{1}{\includegraphics{pinch1.eps}}}
%\rput(1,5){$U_{i-1}$} \rput(9.5,4.7){$U_{i+1}$} \rput(3.2,6.3){$U_i = U_j$}
%\rput[br](5,7){$U_{j-1}$} \rput(4.3,1){$U_{j+1}$}
%\rput[t](3.4,3.5){$x_i$} \rput[t](7.8,3.3){$x_{i+1}$} \rput[l](5.9,4.8){$x_j$} \rput[l](5.6,2.6){$x_{j+1}$}
%\end{pspicture}
%\caption{Original chain $e$ \label{fig:pathe}}
%\end{figure}

%\begin{figure}[h]

%\begin{pspicture}(0,0)(13,7.5)
%\psgrid(0,0)(13,7.5)
%\rput[bl](0,0.5){\scalebox{1}{\includegraphics{pinch2.eps}}}
%\rput(1,5){$U_{i-1}$} \rput(9.5,4.7){$U_{i+1}$} \rput(3.2,6.3){$U_i = U_j$}
%\rput[br](5,7){$U_{j-1}$} \rput(4.3,1){$U_{j+1}$}
%\rput[t](3.4,3.5){$x_i$} \rput[t](7.8,3.3){$x_{i+1}$} \rput[l](5.9,4.8){$x_j$} \rput[l](5.6,2.6){$x_{j+1}$}
%\end{pspicture}
%\caption{Modified chain $e'$ \label{fig:pathep}}
%\end{figure}

Now we can reorder the simplices in $e'$ in the following fashion: \[\begin{split} e \sim e' & = \Big( (x_1 \ x_2) + \ldots + (x_{i-1} \ x_i) + (x_i \ x_{j+1}) + (x_{j+1} \ x_{j+2}) + \ldots + (x_{n-1} \ x_n) \Big) + \\ & + \Big( (x_j \ x_{i+1}) + (x_{i+1} \ x_{i+2}) + \ldots + (x_{j-1} \ x_j) \Big) \end{split}\]

The second term is evidently a simple elementary cycle.
%The first summand is the elementary chain associated to the collection $x_1, \ldots, x_i, x_{j+1}, \ldots, x_n$. The second summand consists of those simplices represented by the ellipsis in the underbraced portion (*) of $e$ together with the single simplex $(x_j \ x_{i+1})$. It is easy to see that it is a simple elementary cycle. 
Now we repeat this process on the first summand of $e'$, and so on until we reach an expression having the form stated in the lemma.
%, and so on. Eventually this will lead to an expression $e \sim e_0 + s_1 + \ldots + s_r$ where $e_0$ is an elementary simple chain (possibly zero if the initial $e$ was a cycle) and each $s_i$ is a simple cycle.
\end{proof}

Notice that if the starting chain $e$ is in fact a cycle then $0 = \partial e = \partial e_{0} + \sum_i \partial s_i = \partial e_0$, and so the $e_0$ summand is also a cycle. Thus, a combination of Lemmas \ref{lem:reduction1} and \ref{lem:red2} yields the following:

\begin{proposition} \label{prop:scc} The $\mathcal{U}$--small elementary simple cycles generate $H_1^{\mathcal{U}}(X;\mathbb Z)$ and, consequently, $H_1^{\mathcal{U}}(X;G)$ for general coefficients $G$.
\end{proposition}

Now we discuss the relation between $H_1^{\mathcal{U}}(X;G)$ and $\check{H}_1(X;G)$ when $X$ is locally connected. \label{pg:refine} It was mentioned earlier on that $H_q^{\mathcal{U}}(X;G)$ essentially ignores all the structure of $X$ below scale $\mathcal{U}$. Suppose that each member $U$ of the covering $\mathcal{U}$ is connected, so that at this scale $X$ has no relevant $0$--dimensional features (these being related to connectedness). It is then natural expect that $H_1^{\mathcal{U}}(X;G)$ is ```smaller'' than $\check{H}_1(X;G)$. Indeed, we shall prove that the projection $\check{H}_1(X;G) \to H_1^{\mathcal{U}}(X;G)$ is surjective.
%This is not generally the case, but we shall prove that every element in $H_1^{\mathcal{U}}(X;G)$ comes from an element in $\check{H}_1(X;G)$.
In other words, as soon as a $1$--cycle is $\mathcal{U}$--small, it can be refined to be $\mathcal{V}$--small for an arbitrary open covering $\mathcal{V}$ in a coherent manner so as to define an element in $\check{H}_1(X;G)$.

Consider a $\mathcal{U}$--small $1$--chain $c = \sum_i k_i \sigma_i$. Each $\sigma_i = (x_i \ y_i)$ is contained in some $U_i \in \mathcal{U}$. Given any open covering $\mathcal V$ of $X$, $\mathcal V|_{U_i} = \{ V \cap U_i: V \in \mathcal V\}$ is an open covering of $U_i$ and since $U_i$ is connected, a standard argument produces an elementary $\mathcal V$--small chain $e_i$ in $U_i$ such that $\partial e_i = y_i - x_i$. Consider the $\mathcal{V}$--small chain $c' := \sum_i k_i e_i$ obtained from $c$ by replacing each $1$--simplex $\sigma_i$ with the chain $e_i$ that connects $x_i$ to $y_i$. We say that $c'$ is a \emph{refinement} of $c$. Observe that $c \sim c'$ in $H_1^{\mathcal{U}}(X;G)$. Indeed, since $\partial e_i = y_i - x_i = \partial \sigma_i$ we have that $e_i - \sigma_i$ is a cycle contained in $U_i$ and so by Proposition \ref{prop:Ucycle} it is nullhomologous. Hence $c' - c = \sum_i k_i (e_i - \sigma_i) \sim 0$ in $H_1^{\mathcal{U}}(X;G)$.

When $X$ is locally connected, open coverings that consist of connected sets constitute a cofinal subfamily of the family of all open coverings. Together with the construction from the previous paragraph, this entails the following:

\begin{proposition} Suppose $X$ is locally connected and there is a cofinal sequence of open coverings (this is the case if $X$ is compact and metric, for example). Assume $\mathcal{U}$ is an open covering of $X$ that consists of connected sets. Then any element of $H_1^{\mathcal{U}}(X;G)$ comes from an element of $\check{H}_1(X;G)$ or, in other words, the projection $\check{H}_1(X;G) \to H_1^{\mathcal U}(X;G)$ is surjective.
\end{proposition}

Suppose for the next definition and lemma that $G = \mathbb Z, \mathbb Q, \mathbb C$. The norm of a $\mathcal U$--small 1-cycle $c = \sum k_i \sigma_i$ is defined as $||c||_1 = \sum |k_i|$. For later purposes we show how to modify the refinement construction just described to control the norm of the refinements:

\begin{lemma} \label{lem:count_refine} Let $\mathcal{U}$ be an open covering that consists of connected sets and consider a $\mathcal{U}$--small chain $c$. Then for any open covering $\mathcal{V}$ there exists a $\mathcal{V}$--small chain $c'$ such that: (i) $c \sim c'$ in $H_1^{\mathcal{U}}(X;G)$ and (ii) ${\|c'\|}_1 \leq {\|c\|}_1 |\mathcal{V}|$, where $|\mathcal{V}|$ denotes the cardinality of $\mathcal{V}$.
\end{lemma}
\begin{proof} Start by constructing $c' = \sum_i k_i e_i$ as before. By Lemma \ref{lem:red2} applied in $U_i$ at scale $\mathcal{V}$ we may write $e_i \sim e_{i0} + s_{i1} + \ldots + s_{ir_i}$ in $H_1^{\mathcal{V}}(U_i;G)$ where each $e_{0i}$ is an elementary simple chain and each $s_{ij}$ is a cycle, and all these are $\mathcal{V}$--small and contained in $U_i$. By Proposition \ref{prop:Ucycle} all the $s_{ij}$ are homologous to zero in $H_1^{\mathcal{U}}(X;G)$, so we have $e_i \sim e_{i0}$ at scale $\mathcal{U}$ for each $i$. Moreover, $\|e_{i0}\|_1 \leq |\mathcal{V}|$ because $e_{i0}$ is an elementary simple chain. Therefore in $H_1^{\mathcal{U}}(X;G)$ we have $c \sim c' := \sum k_i e_{i0}$ and $\|c'\|_1 \leq \sum |k_i| \|e_{i0}\|_1 \leq \sum |k_i| |\mathcal{V}| = \|c\|_1 |\mathcal{V}|$.
\end{proof}

\section{Alexander-Spanier cocycles with a finite image}\label{sec:imagenfinita}

Let $f \colon X \to Y$ be a map, not necessarily continuous, and suppose that $\mathcal{V}$ is an open covering of $Y$ and $\mathcal{U}$ is an open covering of $X$ that refines $f^{-1}(\mathcal{V})$. The usual formula $f_{\sharp}(\xi) := \xi \circ f$ correctly defines a homomorphism $\overline{C}^q_{\mathcal{V}}(Y) \to \overline{C}^q_{\mathcal{U}}(X)$ and, similarly, the prescription \[f_{\sharp}(x_0 \ \ldots \ x_q) := (f(x_0) \ \ldots \ f(x_q))\] for any $\mathcal{U}$--small $q$--simplex extends to a homomorphism $f_{\sharp} \colon S^{\mathcal{U}}_q(X) \to S^{\mathcal{V}}_q(Y)$. These commute with the coboundary and boundary operators, respectively, and therefore induce homomorphisms $f_* \colon H_q^{\mathcal{U}}(X) \to H_q^{\mathcal{V}}(Y)$ and $f^* \colon H^q_{\mathcal{V}}(Y) \to H^q_{\mathcal{U}}(X)$.

%Note that the hypothesis that $\mathcal U$ refines $f^{-1}(\mathcal V)$ is the minimum requirement for the maps $f_* \colon H_q^{\mathcal U}(X) \to H_q^{\mathcal V}(Y)$ and $f^* \colon H^q_{\mathcal V}(Y) \to H^q_{\mathcal U}(X)$ be well defined.
The following proposition estimates how close must two maps $f, g \colon X \to Y$ be in order to induce the same map in homology and cohomology for a fixed choice of open coverings. The condition is a direct translation of the notion of contiguity in the Vietoris complex:

\begin{proposition}\label{1_proposicion_f=g}
Suppose that for every $\mathcal{U}$--small simplex $(x_0 \ \ldots \ x_q)$ in $X$ there exists $V \in \mathcal V$ which contains both simplices $(f(x_0) \ \ldots \ f(x_q))$ and $(g(x_0) \ \ldots \ g(x_q))$. Then $f_* = g_* \colon H_q^{\mathcal{U}}(X) \to H_q^{\mathcal{V}}(Y)$ and similarly for cohomology.
\end{proposition}
\begin{proof} Observe that the assumption implies that $f$ and $g$ carry $\mathcal{U}$--small simplices to $\mathcal{V}$--small simplices, so $f_*$ and $g_*$ are well defined. For every $\mathcal{U}$--small simplex $(x_0\ \ldots \ x_q)$ in $X$ define $$D (x_0 \ \ldots \ x_q) :=\sum_{i=0}^{q}(-1)^i (f(x_{0}) \ \ldots \ f(x_{i}) \ g(x_{i}) \ \ldots \ g(x_{q})).$$ This is a $(q+1)$--chain in $Y$ which is $\mathcal{V}$--small because of the hypothesis of the proposition. Extending $D$ linearly to $D \colon S_q^{\mathcal{U}}(X) \rightarrow S_{q+1}^{\mathcal{V}}(Y)$ we obtain a prism operator which is easily checked to satisfy $\partial D + D \partial = g_{\sharp} - f_{\sharp}$. Thus $D$ is a chain homotopy between $g_{\sharp}$ and $f_{\sharp}$ and so the two induce the same homomorphism $f_* = g_* \colon H_q^{\mathcal{U}}(X) \to H_q^{\mathcal{U}}(Y)$. The argument for cohomology is analogous.
\end{proof}

Now we describe a construction which provides some sort of ``sampling'' of a space $X$. It will also justify why it is interesting to consider maps that are not necessarily continuous. We first assume that $X$ is paracompact, although we will subsequently specialize to the case when $X$ is compact. Let $\mathcal{V}$ be an open covering of $X$ and let $\mathcal{U}$ be a star refinement of $\mathcal{V}$. This means that (*) for every $U \in \mathcal{U}$ there exists $V \in \mathcal{V}$ that contains every element of $\mathcal U$ that intersects $V$, including $U$ itself. The existence of such a refinement $\mathcal U$ can be easily established when $X$ is compact metric using a Lebesgue number for $\mathcal V$, but in general it is equivalent to the assumption that $X$ be paracompact. Assume that every member of $\mathcal{U}$ is nonempty.

Write $\mathcal{U} = \{U_i : i \in I\}$ and endow the index set $I$ with a well-ordering. For each $i$ pick a point $u_i \in U_i$ and consider the map 
\begin{align*}
p:X&\to X\\
  x&\longmapsto u_{i} \ \ \mbox{$i \in I$ smallest index with $x\in U_{i}$.}
\end{align*}
The identity map on $X$, $\mathrm{id}_X$, and $p$ verify the hypothesis of Proposition \ref{1_proposicion_f=g}. We conclude that they define the same map $H_q^{\mathcal{U}}(X) \to H_ q^{\mathcal{V}}(X)$ (and similarly in cohomology). That is, we have
%\textcolor{red}{LINEAS PRESCINDIBLES}
%Let $(x_0\ \ldots \ x_q)$ be contained in $U \in \mathcal{U}$ and write $p(x_k) = u_{i(k)}$ for each $k$. By the definition of $p$ we have $x_k \in U_{i(k)}$, so $U_{i(k)} \cap U \neq \emptyset$ for each $k$. Thus by (*) there exists $V \in \mathcal{V}$ that contains all the $u_{i(k)}$ and in particular all the $u_{i(k)} = p(x_k)$ and also all the $x_k$. In other words, it contains both $(x_0 \ \ldots \ x_q)$ and $(p(x_0) \ \ldots \ p(x_q))$ or, in the language of simplicial complexes, the maps induced by $p$ and $\mathrm{Id}_X$ from $V(\mathcal U)$ to $V(\mathcal V)$ are contiguous. By the previous proposition $p$ and $\mathrm{Id}_X$ define the same map $H_q^{\mathcal{U}}(X) \to H_ q^{\mathcal{V}}(X)$ (and similarly in cohomology).
%Now, the map induced by the identity is precisely the bonding map of the inverse (direct) system that defines $\check{H}_q(X)$ (respectively $\check{H}^q(X)$).
%That is, we have
\[p_* = j_{\mathcal{U},\mathcal{V}} \colon H^{\mathcal{U}}_q(X) \to H^{\mathcal{V}}_q(X) \quad \text{ and } \quad p^{*} = \pi_{\mathcal{V},\mathcal{U}} :H_{\mathcal{V}}^{q}(X)\to H_{\mathcal{U}}^{q}(X).\]

From now on we concentrate on the case when $X$ is compact. Then one can always replace $\mathcal{U}$ with a finite subcover and this is still a star refinement of $\mathcal{V}$. Let $X_0 = \{u_i : i \in I\}$. This is a ``finite sample'' of $X$ \label{pag:samples}. Evidently the image of $p$ is contained in $X_0$, so we can write $p = i \circ p_0$ where $p_0$ is the map $p$ whose target space is restricted to $X_0$ and $i \colon X_0 \subseteq X$ denotes the inclusion. Hence we can factor $p_*$ and $p^*$ through $H_*^{\mathcal{V}}(X_0)$ and $H^*_{\mathcal{V}}(X_0)$, respectively, to obtain the following commutative diagrams: \[\xymatrix{H_q^{\mathcal{U}}(X) \ar[r]^{j_{\mathcal{U},\mathcal{V}}} \ar[rd]_{(p_0)_*} & H_q^{\mathcal{V}}(X) \\  & H_q^{\mathcal{V}}(X_0) \ar[u]^{i_*} } \quad \quad \text{ and } \quad  \quad \xymatrix{H^q_{\mathcal{U}}(X) & H^q_{\mathcal{V}}(X) \ar[l]_{\pi_{\mathcal{V},\mathcal{U}}} \ar[d]_{i^*} \\ & H^q_{\mathcal{V}}(X_0) \ar[ul]^{(p_0)^*} }\] The useful feature of these diagrams is that the homology and cohomology groups of $X_0$ are entirely finitistic in nature since $X_0$ is finite. The number of $q$--simplices in $X_0$ is finite, with a very crude bound being $|X_0|^{q+1}$. This places an upper bound on the rank of $H_q^{\mathcal{V}}(X_0)$. Similarly, any Alexander--Spanier cochain defined on $X_0$ necessarily has a finite image with no more than $|X_0|^{q+1}$ elements. Thus we have the following proposition, originally due to Keesee \cite{keesee}.

\begin{proposition}\label{prop:imagenfinita}
Let $X$ be compact. Then, every element of $\check{H}^*(X)$ has a representative cocycle whose image is finite.
\end{proposition}

\begin{proof}
Let $z \in \check{H}^q(X)$. Recall from Proposition \ref{prop:limUcohomologia} that the \v{C}ech cohomology group is the direct limit of the $\mathcal V$--cohomology groups, $H_{\mathcal V}^q(X)$, where $\mathcal V$ ranges over all open coverings of $X$. By compactness, we can restrict the limit to finite open coverings. Thus, there exists a finite open covering $\mathcal V$ of $X$ and $z_{\mathcal{V}}\in H_{\mathcal{V}}^{q}(X)$ such that 
\begin{align*}
H_{\mathcal{V}}^{q}(X)&\stackrel{\pi_{\mathcal{V}}}{\to}\check{H}^{q}(X)\\
    z_{\mathcal{V}}&\longmapsto z
\end{align*}

Let $z_{\mathcal{U}} := \pi_{\mathcal{V},\mathcal{U}}(z_{\mathcal{V}}) \in H^q_{\mathcal{U}}(X)$, which in the limit also represents the element $z$ since $\pi_{\mathcal{U}} \circ \pi_{\mathcal{V},\mathcal{U}} = \pi_{\mathcal{V}}$. By the commutative diagram above we then have $z_{\mathcal{U}} = (p_0)^*  i^*(z_{\mathcal{V}})$, and evidently $i^*(z_{\mathcal{V}})$ is represented by some cocycle $\xi$ with a finite image. Thus $z_{\mathcal{U}}$ and also $z$ are represented by the cocycle $(p_0)^{\sharp} \xi$, which also has a finite image.
\end{proof}

\begin{remark}\label{rmk:cohomologiafingen}
The previous arguments trivially show that $H^*_{\mathcal U}(X)$ is finitely generated because every open covering $\mathcal U$ is a star refinement of the trivial open covering $\{X\}$.
\end{remark}

%============================================================
\section{Definition of the integral}\label{sec:integral}
%============================================================

In the usual definition of cohomology as the homology of a dual chain complex there is an obvious manner in which cohomology acts on homology: by evaluation of cochains on chains. The descriptions of \v{C}ech homology and cohomology considered in this paper are not so obviously dual to each other; however, it is possible to define a pairing between cohomology and homology classes in terms of an integral reminiscent of de Rham's theory of integration of differential forms over simplices. We first devote a few lines to a heuristic motivation of the integral.

\subsection{Motivation} Consider a space $X$. Borrowing the language of differential geometry just for heuristic purposes, we can think of a $q$--cochain $\xi$ as a $q$--differential form and of a (formal) $q$--simplex $\sigma = (x_0 \ \ldots\ x_q)$ as a base point $x_0$ together with $q$ ``vectors'' $\vv{x_0x_i}$ that are ``approximately tangent'' to $X$ at $x_0$. Then the evaluation $\xi(\sigma) = \xi(x_0, \ldots, x_q)$ can be thought of as the evaluation of the form $\xi_{x_0}$ over the tuple of $q$ tangent vectors $(\vv{x_0x_1},\ldots,\vv{x_0x_q})$. This analogy suggests how to define the integral of a cohomology class $z \in \check{H}^q(X)$ over a homology class $\gamma \in \check{H}_q(X)$: one takes a cocycle $\xi$ that represents $z$ and an approximation $c = \sum k_i \sigma_i$ of $\gamma$, computes the ``Riemann sum'' $\sum k_i \xi(\sigma_i)$ and then lets the size of the simplices of the approximation go to zero. There is no \emph{a priori} reason why the limit should exist; however, the cocycle condition $\delta \xi = 0$ ensures that it does.

To explain this in more detail let us consider a particularly simple situation. Suppose that $\gamma \colon [0,1] \to X$ is a continuous closed path in $X$ and $\xi$ is a 1--cocycle. To integrate $\xi$ over $\gamma$ as suggested above we would consider a partition $\{t_i\}$ of $[0,1]$ and then evaluate $\xi$ over the ``approximate tangent vectors'' $\vv{\gamma(t_i) \gamma(t_{i+1})}$ to obtain the Riemann sum $\sum \xi(\gamma(t_i),\gamma(t_{i+1}))$. Then we would take progressively finer partitions $\{t_i\}$ of $[0,1]$ letting their diameters go to zero. What happens if we refine a partition by inserting $t_* \in (t_i,t_{i+1})$? The Riemann sum is modified by 

\centerline{$\xi(\gamma(t_i),\gamma(t_*)) + \xi(\gamma(t_*),\gamma(t_{i+1})) - \xi(\gamma(t_i),\gamma(t_{i+1})) = \delta \xi (\gamma(t_i), \gamma(t_*), \gamma(t_{i+1}))$.}

By the cocycle condition, this expression vanishes when the points are close enough. Thus, when the partition is fine enough further refinement do not change the value of the Riemann sum and, as a consequence, the limit of the sums as the diameter of the partition tends to zero trivially exists.

The same ideas also translate directly to an integration over \v{C}ech 1--homology classes $\gamma$: in that case the progressive refinement of the partition $\{t_i\}$ is actually part of the definition of $\gamma$, afforded by taking representatives $\gamma^{\mathcal{U}}$ of $\gamma$ that are $\mathcal{U}$--small for progressively fine coverings $\mathcal{U}$. 

%It should be kept in mind that in order to define integration along homology classes, the integral along two homologous chains must give equal results. As in the differentiable setting, the key to this property is a Stokes' formula that ensures that the integral of a cocycle along a boundary vanishes.

\subsection{Formal definition}

Now we define formally the integral of an Alexander-Spanier cohomology class $z \in \check{H}^q(X)$ over a \v{C}ech homology class $\gamma \in \check{H}_q(X)$. For the definition to make algebraic sense we must be able to multiply coefficients in homology with coefficients in cohomology. This can be ensured in many ways; for instance by taking homology with $\mathbb{Z}$--coefficients and cohomology with coefficients in an arbitrary group. We will not pursue the maximum generality in this regard, so we just require that the group of coefficients $G$ be a ring itself, and in fact will soon specialize to the case when $G$ is a field $\mathbb{K}$.

We need some preliminaries. Let $\xi \colon X^{q+1} \to G$ be a $q$--cochain and let $\sigma = (x_0 \ \ldots \ x_q)$ be a $q$--simplex. We can evaluate $\xi$ over $\sigma$ to obtain an element in $G$ in the obvious way, setting $\xi(\sigma) := \xi(x_0,\ldots,x_q)$. If $c = \sum k_i \sigma_i$ is a $q$--chain we define the evaluation of $\xi$ over $c$ by extending linearly the definition above: $\xi(c) := \sum k_i \xi(\sigma_i)$. Denoting by $\xi( \cdot )$ the standard evaluation of $\xi$ on a tuple of points and the evaluation of $\xi$ on a simplex or a chain is a slight abuse of notation but should not cause any confusion. The following algebraic property is key to what follows and can be easily deduced from the definitions.

\begin{lemma}[Stokes' lemma] Let $\xi$ be a $(q-1)$--cochain and $c = \sum k_i \sigma_i$ a $q$--chain. Then $(\delta \xi)(c) = \xi(\partial c)$.
\end{lemma}

Now let $z \in \check{H}^q(X)$ and $\gamma \in \check{H}_q(X)$. Let $\xi$ be a $q$--cocycle that represents $z$ and let $\mathcal{U}$ be an open covering of $X$ such that $\delta \xi$ is zero over each member of $\mathcal{U}$. Let $c$ be a $\mathcal{U}$--small representative of $\gamma$. We define the integral of $z$ over $\gamma$ as \begin{equation} \label{eq:integral1} \int_{\gamma} z := \xi(c).\end{equation}
%More explicitly, if we write $c = \sum_i k_i (x_{i0} \ x_{i1} \ \ldots \ x_{i q})$ the above definition gives \[\int_{\gamma} z = \sum_i k_i \, \xi(x_{i0},x_{i1},\ldots,x_{i q})\] which agrees precisely with our heuristic motivation about ``Riemann sums of differential forms''.
There are several choices involved in this definition and to justify that the definition is correct we need to prove the following:

\begin{theorem} \label{teo:welldefined} The right hand side of \eqref{eq:integral1} is independent of the choices of $c$, $\mathcal{U}$, and $\xi$.
\end{theorem}
\begin{proof} (i) Let $c_1$ and $c_2$ be two $\mathcal{U}$--small chains that represent $\gamma$. Then $c_1 - c_2 = \partial d$ for some $\mathcal{U}$--small chain $d$. By the linearity of evaluation and Stokes' lemma we have \[\xi(c_1) - \xi(c_2) = \xi(c_1 - c_2) = \xi(\partial d) = (\delta \xi)(d)\] and the latter term vanishes because $\delta \xi$ is zero over every member of $\mathcal{U}$ and $d$ is $\mathcal{U}$--small. Thus the right hand side of Equation \eqref{eq:integral1} does not depend on the particular $\mathcal{U}$--small chain $c$ chosen to represent $\gamma$.

(ii) Let $\mathcal{V}$ be another open covering such that $\delta \xi$ is locally zero over $\mathcal{V}$ and suppose first that $\mathcal{V}$ refines $\mathcal{U}$. Let $c^{\mathcal{V}}$ and $c^{\mathcal{U}}$ be representatives of $\gamma$ that are $\mathcal{V}$-- and $\mathcal{U}$--small respectively. Observe that $\gamma^{\mathcal V} = [c^{\mathcal{V}}]$ gets mapped to $\gamma^{\mathcal U} = [c^{\mathcal{U}}]$ in $H^{\mathcal{U}}_*(X)$ by the bonding morphism $H_q^{\mathcal{V}}(X) \rightarrow H_q^{\mathcal{U}}(X)$. Thus there exists a $\mathcal{U}$--small chain $d$ such that $c^{\mathcal{V}} - c^{\mathcal{U}} = \partial d^{\mathcal{U}}$ in the group of $\mathcal{U}$--chains. Again by linearity of evaluation and Stokes' lemma we have \[\xi(c^{\mathcal{U}}) - \xi(c^{\mathcal{V}}) = \xi(c^{\mathcal{U}} - c^{\mathcal{V}}) = \xi(\partial d^{\mathcal{U}}) = (\delta \xi)(d^{\mathcal{U}}) = 0,\]
so $\xi(c^{\mathcal{U}}) = \xi(c^{\mathcal{V}})$. In the general case when $\mathcal{V}$ does not refine $\mathcal{U}$, one simply goes through the common refinement $\mathcal{U} \vee \mathcal{V} := \{U \cap V : U \in \mathcal{U}, V \in \mathcal{V}\}$.% This refines both $\mathcal{U}$ and $\mathcal{V}$, and so by the previous argument $\xi(c^{\mathcal{U}}) = \xi(c^{\mathcal{U} \vee \mathcal{V}}) = \xi(c^{\mathcal{V}})$, where $c^{\mathcal{U} \vee \mathcal{V}}$ is any $(\mathcal{U} \vee \mathcal{V})$--small representative of $\gamma$.

(iii) Finally, to check that the right hand side of Equation \eqref{eq:integral1} is independent of $\xi$, let $\xi_1$ and $\xi_2$ be two representatives of $z$. Then there exist $\eta$ and an open covering $\mathcal{U}$ such that $\xi_1 - \xi_2 = \delta \eta$ over each member of $\mathcal{U}$ and also $\delta \xi_1$ and $\delta \xi_2$ vanish over each member of $\mathcal{U}$. Let $c$ be a $\mathcal{U}$--small representative of $\gamma$. Then \[\xi_1(c) - \xi_2(c) = (\xi_1-\xi_2)(c) = (\delta \eta)(c) = \eta(\partial c) = \eta(0) = 0,\]
and $\xi_1(c) = \xi_2(c)$ as desired.
%where in the second equality we have used that $c$ is a $\mathcal{U}$--small chain and $\xi_1-\xi_2 = \delta \eta$ over each member of $\mathcal{U}$, and in the second to last equality we have used that $c$ is a cycle.  
\end{proof}

One would expect that the integral be linear in both $z$ and $\gamma$, and this is certainly the case. For later convenience we reformulate this as follows. The integral can be considered as a map \[\int : \check{H}^q(X) \times \check{H}_q(X) \to G\] and then we have:

\begin{proposition} The integral is a bilinear form.
\end{proposition}
\begin{proof} This is straightforward using the freedom in $c$, $\mathcal{U}$ and $\xi$ afforded by the previous theorem. Just to illustrate this let us show that $\int_{\gamma} (z + z') = \int_{\gamma} z + \int_{\gamma} z'$ for any $\gamma\in\check{H}_{q}(X)$ and $z,z'\in\check{H}^{q}(X)$. Let $\xi$ be a cocycle that represents $z$ and let $\mathcal{U}$ be an open covering such that $\delta \xi$ is zero over each member of $\mathcal{U}$. Let $\xi'$ and $\mathcal{U}'$ be analogous objects for $z'$. Then $\xi + \xi'$ represents $z+z'$ and its coboundary is zero over each member of the common refinement $\mathcal{U} \vee \mathcal{U}'$. Observe that $\delta \xi$ and $\delta \xi'$ are still zero over each member of $\mathcal{U} \vee \mathcal{U}'$, and so we may use the latter covering to compute the three integrals $\int_{\gamma} z$, $\int_{\gamma} z'$ and $\int_{\gamma} (z+z')$ by choosing a representative $c$ of $\gamma$ which is $(\mathcal{U} \vee \mathcal{U}')$--small. Then clearly $$\int_{\gamma}(z+z')=(\xi+\xi')(c)=\xi(c)+\xi'(c)=\int_{\gamma}z+\int_{\gamma}z'.$$
\end{proof}

It is convenient to observe that the definition of the integral makes sense not only in the limit of \v{C}ech homology and cohomology, but already at scale $\mathcal{U}$ for any open covering $\mathcal{U}$ of $X$. Indeed, given $z \in H_{\mathcal{U}}^q(X)$ and $\gamma \in H_q^{\mathcal{U}}(X)$, one can pick a cocycle $\xi$ which represents $z$ and satisfies $\partial \xi = 0$ over every $\mathcal{U}$--small simplex and a $\mathcal{U}$--small chain $c$ which represents $\gamma$ and define \[\int_{\gamma} z := \xi(c).\] The proof that this is independent of $\xi$ and $c$ is completely analogous to steps (i) and (iii) in the proof of Theorem \ref{teo:welldefined}. In particular, the integral in this generalized sense is still bilinear.

\begin{remark}\label{rmk:scales}
There is no need to record $\mathcal{U}$ in the notation or even to distinguish this notationally from the integral defined at the beginning of this section, and the reason is the following. A cohomology class $z_{\mathcal U} \in H^*_{\mathcal U}(X)$ represented by a cocycle $\xi$ can be viewed as a cohomology class $z_{\mathcal V}$ at scale $\mathcal V$ for any refinement $\mathcal V$ of $\mathcal U$ or as an element $z$ of $\check{H}^*(X)$ represented by the same cocycle $\xi$. If we want to integrate $z$ over a \v{C}ech homology class $\gamma = (\gamma^{\mathcal U})$ all we have to do is take a $\mathcal U$--small cycle $c$ that represents $\gamma$ and return $\xi(c)$. Exactly the same computation can be carried out to compute the integral of $z_{\mathcal U}$ along $\gamma^{\mathcal U}$ because $c$ is a representative of $\gamma^{\mathcal U}$. Note that we could have chosen also a representative $c'$ of $\gamma^{\mathcal V}$ for some refinement $\mathcal V$ to compute $\int_{\gamma^{\mathcal U}} z_{\mathcal U}$ because $c$ and $c'$ are homologous as $\mathcal U$--small cycles.
\end{remark}

The integral also behaves well under the action of continuous maps:

\begin{lemma}\label{lemma_evaluation_map}
Let $f\colon X\to Y$ be a map. 
Then 
$$\xi(f_{\sharp}c)=f^{\sharp}\xi(c)$$
where $\xi$ is a $q$--cochain in $Y$ and $c$ is a $q$--chain in $X$.
\end{lemma}

\begin{lemma}\label{lem:integral_map}
Let $f\colon X\rightarrow Y$ be a map and $\mathcal U$ and $\mathcal V$ open coverings of $X$ and $Y$, respectively, such that $\mathcal U$ refines $f^{-1} \mathcal V$.
Then 
$$\int_{\gamma}f^{*}z=\int_{f_{*}\gamma}z$$
where $z\in H^{q}_{\mathcal V}(Y)$ and $\gamma\in H_{q}^{\mathcal U}(X)$.
\end{lemma}

\begin{proof}
Let  $z\in H^{q}_{\mathcal V}(Y)$ and $\gamma\in H_{q}^{\mathcal U}(X)$.
Let $\xi$ be a $q$--cocycle that represents $z$. Since $\delta \xi$ is zero over each member of $\mathcal V$ and $\mathcal U$ refines $f^{-1}\mathcal V$, we conclude that $\delta f^{\sharp}\xi$ vanishes over each member of $\mathcal U$. Letting $c$ be a $\mathcal{U}$--small representative of $\gamma$, since $f_{\sharp}c$ is a $\mathcal{V}$--small representative of $f_{*}\gamma$ we have that
$$\int_{\gamma}f^{*}z=f^{\sharp}\xi(c)=\xi(f_{\sharp}c)=\int_{f_{*}\gamma}z.$$
%where the second equality is due to the Lemma \ref{lemma_evaluation_map}.
\end{proof}

%=========================================================
\section{Nondegeneration of the integral}\label{sec:nondegenerate}
%=========================================================

Throughout all this section the \emph{coefficients are taken in a field $\mathbb K$ which we shall omit from the notation}.

To motivate the forthcoming discussion let us recall a well known result in multivariable calculus. Let a vectorfield $V$ be defined on some open subset $X$ of $\mathbb{R}^3$ and assume that the integral of $V$ along any closed curve contained in $X$ vanishes. Then (under suitable smoothness assumptions) $V$ is the gradient of some function $f$; a so-called potential for the vectorfield. The standard proof constructs explicitly the potential $f$ as follows. One fixes some reference point $p_0 \in X$ and then for each $p \in X$ chooses a path $\gamma_p$ contained in $X$ that joins $p_0$ with $p$ (if $X$ is not connected this has to be done on each connected component). Then $f(p)$ is defined as the integral of the vectorfield $V$ along the path $\gamma_p$. The assumption that the integral of $V$ along any closed curve is zero ensures that $f(p)$ is independent of the path $\gamma_p$, and checking that the gradient of $f$ is precisely $V$ is then just a matter of simple differential calculus.

The same construction can be used, with appropriate modifications, in our setting. Suppose $X$ is an arbitrary space. Let $\mathcal{U}$ be an open covering of $X$ and let $\overline{\xi} \in \overline{C}^1_{\mathcal U}(X)$, i.e. $\xi$ is a 1--cochain such that $\delta \xi$ vanishes over every member of $\mathcal U$ and so $\overline{\xi}$ is a representative of a cohomology class $z \in H^1_{\mathcal U}(X)$. Assume further that $\xi$ has the property that for every $\mathcal{U}$--small cycle $c$, the evaluation $\xi(c)$ vanishes. Then there exists a ``potential'' for $\xi$; that is, a 0--cochain $V$ such that $\delta V = \xi$ on each element of $\mathcal{U}$. To define $V$ we can fix a reference point $p_0 \in X$ and for any other point $p$ that can be joined to $p_0$ by a $\mathcal U$--small 1-chain $\sigma$ ($\partial \sigma = p - p_0$) set $V(p) = \xi(\sigma)$, which is in spirit $\int_{\sigma} \xi$ although we prefer not to use this notation because $\sigma$ is not a cycle. Repeat the procedure until the definition reaches every point of $X$ to account for multiple ``$\mathcal U$--components''.

Note that the existence of $V$ that solves $\delta V = \xi$ implies that $z = 0$. Also, the assumption that $\xi(c)$ vanishes for each $\mathcal{U}$--small cycle is equivalent to the condition that the integral of $z$ over any $\gamma \in H_1^{\mathcal{U}}(X)$ vanishes.
In other words, we have the following result for $q = 1$:

\begin{equation}\label{eqn:injective1}
\int_{\gamma} z = 0 \quad \forall \gamma \in H_q^{\mathcal{U}}(X) \enskip \Rightarrow \enskip z = 0 \in H^q_{\mathcal U}(X).
\end{equation}

\begin{proposition}\label{prop:injective1}
%Let coefficients be taken in a field $\mathbb{K}$.
The statement in (\ref{eqn:injective1}) is true for every $q \geq 0$.
\end{proposition}
\begin{proof} %The case $q = 0$ is immediate because any point $x$ defines a class $[x]$ in $H_0^{\mathcal U}(X)$ and the integration over $[x]$ is simply the evaluation of the cocycle at $x$.

Let us construct the higher dimensional analogue of the ``potential''. Let $\overline{\xi} \in \overline{C}^q_{\mathcal U}(X)$ be a cocycle which represents $z$. The goal is to find a cochain $\overline{\eta} \in \overline{C}^{q-1}_{\mathcal U}(X)$ that satisfies $\delta \overline{\eta} = \overline{\xi}$; that is, $\delta \eta(c) = \xi(c)$ for every $\mathcal{U}$--small $q$--chain $c$. %Notice that here we are identifying cochains with their linear extensions to operators on the group of chains, as described in p. \pageref{rem:cochains_ext}. \textcolor{red}{OJO CON LA REFERENCIA}

Let $B_{q-1}^{\mathcal{U}}(X) \subseteq S_{q-1}^{\mathcal{U}}(X)$ be the subspace of boundaries; that is, the set of $\mathcal{U}$--small $(q-1)$--chains that bound a $\mathcal{U}$--small $q$--chain. In turn, $S_{q-1}^{\mathcal{U}}(X)$ is a subspace of $S_{q-1}^{\{X\}}(X)$. Notice that this last group contains all formal $(q-1)$--chains in $X$ regardless of their size.

Take a basis $\mathcal{B} = \{d_i\}$ of $B_{q-1}^{\mathcal{U}}(X)$. For each $d_i$ let $c_i$ be a $\mathcal{U}$--small $q$--chain such that $d_i = \partial c_i$. There exists a unique linear map $\eta$ defined on $B_{q-1}^{\mathcal{U}}(X)$ such that $\eta(d_i) = \xi(c_i)$. Extend arbitrarily this to another linear map (again denoted by $\eta$) defined on all of $S_{q-1}^{\mathcal{U}}(X)$, and then again to a linear map defined on all of $S_{q-1}^{\{X\}}(X)$. In particular $\eta(x_0,\ldots,x_{q-1}) = \eta(x_0 \ \ldots \ x_{q-1})$ is well defined for every tuple of points in $X$.

We claim that $\delta \eta(c) = \xi(c)$ for every $\mathcal{U}$--small $q$--chain $c$. Let $c$ be such a chain. Its boundary is an element of $B_{q-1}^{\mathcal{U}}(X)$ and can therefore be written as a finite linear combination of elements of $\mathcal{B}$, say $\partial c = \sum k_i d_i$. Then \[\delta \eta(c) = \eta(\partial c) = \sum k_i \eta(d_i) = \sum k_i \xi(c_i) = \xi( \sum k_i c_i) = \xi(c')\] where $c':= \sum k_i c_i$ satisfies $\partial c' = \partial c$. Therefore $\gamma = [c'-c]$ is an element of $H_q^{\mathcal{U}}(X)$. By assumption $0 = \int_{\gamma} z = \xi(c'-c)$, so $\xi(c') = \xi(c)$ and the equality above leads to $\delta \eta(c) = \xi(c)$, as was to be shown.
\end{proof}

It is also natural to ask if a dual result holds, that is, whether an element $\gamma \in H_q^{\mathcal{U}}(X)$ which satisfies $\int_{\gamma} z = 0$ for every $z \in H^q_{\mathcal{U}}(X)$ must necessarily be zero.

\begin{proposition} \label{prop:injective2}
%Let coefficients be taken in a field $\mathbb{K}$.
The statement dual to (\ref{eqn:injective1}) is true for every $q \geq 0$.
\end{proposition}
\begin{proof}
%For the case $q = 0$ it suffices to note that the $0$--cocycles are the functions which are constant in the ``$\mathcal U$--components'' ($C$ is a $\mathcal U$--component if contains or is disjoint to any member of $\mathcal U$).

%We have to prove that if $\int_{\gamma}z = 0$ for every $z \in H^q_{\mathcal{U}}(X)$ then $\gamma = 0$.
By contradiction, at the level of chains and cochains, given a $\mathcal{U}$-small cycle $c = \sum_{i=0}^n k_i \sigma_i$ that is not a boundary we need to find a cocycle $\xi \in C^q_{\mathcal U}(X)$ such that $\xi(c) \neq 0$. Assume without loss of generality that no linear combination of the simplices $\sigma_i$ yields a nontrivial boundary.

Define $\nu \colon \mathrm{span}\{\sigma_i\} \oplus B_q^{\mathcal U}(X) \to \mathbb K$ by $\nu(\sigma_0) = 1$, $\nu(\sigma_i) = 0$ if $i \neq 0$ and $\nu \equiv 0$ on $B_q^{\mathcal U}(X)$. This map extends linearly to the vector space $S_q^{\mathcal U}(X)$ of $\mathcal U$-small chains. The extension, which we denote by $\nu$ as well, satisfies $\delta \nu(d) = \nu(\partial d) = 0$ for every $\mathcal U$-small $(q+1)$--simplex $d$ so it may be seen as a cocycle that represents a cohomology class in $H^q_{\mathcal U}(X)$. From the definition of $\nu$ we obtain that $\nu(c) = \sum_i k_i \nu(\sigma_i) = k_0 \neq 0$, as desired.
 
\end{proof}

Let us formulate the previous results in algebraic terms. Regard again the integral as a bilinear pairing \[\int \colon H^q_{\mathcal{U}}(X) \times H_q^{\mathcal{U}}(X) \to \mathbb K.\] This then defines linear maps from one of its source groups to the dual of the other in a canonical fashion:

\begin{minipage}{0.5\textwidth}

\[\begin{array}{rcc}
H^q_{\mathcal{U}}(X) & \to & {\rm Hom} (H_q^{\mathcal{U}}(X),\mathbb K) \\
z & \longmapsto & I(z,\cdot) \colon \gamma \mapsto \displaystyle\int_{\gamma}z
\end{array}
\]
\end{minipage}
\begin{minipage}{0.5\textwidth}
\[\begin{array}{rcc}
H_q^{\mathcal{U}}(X) & \to & {\rm Hom} (H^q_{\mathcal{U}}(X), \mathbb K) \\
\gamma & \longmapsto & I(\cdot, \gamma) \colon z \mapsto \displaystyle\int_{\gamma}z
\end{array}
\]
\end{minipage}

\noindent We shall call these the canonical maps associated to the bilinear form $\int$ at scale $\mathcal U$.

\begin{proposition} \label{prop:nondeg_U}
%Let coefficients be taken in a field $\mathbb{K}$.
Let $\mathcal{U}$ be a finite open covering of $X$. Then, the canonical maps at scale $\mathcal U$ are isomorphisms.
\end{proposition}
\begin{proof} The conclusion follows from Propositions \ref{prop:injective1} and \ref{prop:injective2} and the fact that homology and cohomology groups $H^{\mathcal{U}}_q(X;\mathbb{K})$ and $H_{\mathcal{U}}^q(X;\mathbb{K})$ are finite dimensional by Remarks \ref{rem:f_gen} and \ref{rmk:cohomologiafingen} (in fact, if one group is finite dimensional then so is the other). 
\end{proof}

One may wonder whether the proposition above holds true not only at a specific scale $\mathcal{U}$ but also in the limit. That is, if we now consider the integral as a bilinear pairing $\int : \check{H}^q(X) \times \check{H}_q(X) \to \mathbb{K}$, are the canonical maps isomorphisms? The following theorem provides an answer for compact spaces. The assumption on compactness is needed to have a cofinal system of finite open coverings on $X$.

\begin{theorem} \label{teo:nondeg} Let $X$ be compact. Then:
\begin{itemize}
	\item[(1)] The canonical map \[\begin{array}{ccc} \check{H}_q(X) & \to & {\rm Hom}(\check{H}^q(X),\mathbb{K}) \\ \gamma & \longmapsto & I(\cdot,{\gamma}) \end{array}\] is an isomorphism.
	\item[(2)] The canonical map \[\begin{array}{ccc} \check{H}^q(X) & \to & {\rm Hom}(\check{H}_q(X),\mathbb{K}) \\ z & \longmapsto & I(z,\cdot) \end{array}\] is injective. It is an isomorphism if either one of $\check{H}_q(X)$ or $\check{H}^q(X)$ are finite dimensional (in which case both are).
\end{itemize}
\end{theorem}

The theorem is an algebraic consequence of Proposition \ref{prop:nondeg_U} and its proof is perhaps less cumbersome when formulated as such. The construction of the integral, in the abstract, can be described as follows. We have a direct system of vector spaces $\{V_i\}$ (the cohomology groups $H^q_{\mathcal{U}}(X)$) and an inverse system of vector spaces $\{W_i\}$ (the homology groups $H_q^{\mathcal{U}}(X)$) both indexed over the same set $F$. We shall denote by $\alpha_{ij} \colon V_i \to V_j$ and by $\beta_{ij} \colon W_j \to W_i$ the bonding maps of the systems and write $\alpha_i \colon V_i \to \varinjlim V_i$ and $\beta_i \colon \varprojlim W_i \to W_i$ for the canonical maps that relate each term of the system with the limit. For each index $i$ there is a bilinear form $B_i \colon V_i \times W_i \to \mathbb{K}$ which is compatible with the bonding maps $\alpha_{ij}$ and $\beta_{ij}$ in the sense that whenever $j \geq i$ the following holds:
\begin{equation} \label{eq:intertwine} B_j(\alpha_{ij}(v_i), w_j) = B_i(v_i,\beta_{ij}(w_j)) \quad \text{for every } v_i \in V_i, w_j \in W_j. \end{equation}
This is the abstract formalization of Remark \ref{rmk:scales}. Finally, the integral is a bilinear map $B \colon (\varinjlim V_i) \times (\varprojlim W_i) \to \mathbb{K}$ defined as follows: to compute $B(v,w)$ we find an index $i$ big enough so that there exists $v_i \in V_i$ such that $\alpha_i(v_i) = v$ and then set
\begin{equation}\label{eq:Binta} B(v,w) := B_i(v_i,\beta_i(w))\end{equation}
or equivalently
\begin{equation*}%\label{eq:Bintb}
B(\alpha_i(v_i),w) = B_i(v_i,\beta_i(w)).
\end{equation*}
In general this does not yield a well defined map $B$ because the right hand side depends on the index $i$, but in the case of the integral it does, see Remark \ref{rmk:scales}.

\begin{proposition}\label{prop:nondegabstract} In the setting just described, assume that the canonical maps associated to each $B_i$ are isomorphisms (and that Equation \eqref{eq:Binta} correctly defines $B$). Then the following hold:
\begin{itemize}
	\item[(1)] The canonical map \[\begin{array}{ccc} \varprojlim W_i & \to & {\rm Hom}(\varinjlim V_i,\mathbb{K}) \\ w & \longmapsto & B(\cdot,w) \end{array}\] is an isomorphism.
	\item[(2)] The canonical map \[\begin{array}{ccc} \varinjlim V_i & \to & {\rm Hom}(\varprojlim W_i,\mathbb{K}) \\ v & \longmapsto & B(v,\cdot) \end{array}\] is injective. It is an isomorphism if either one of $\varinjlim V_i$ or $\varprojlim W_i$ are finite dimensional (in which case both are).
\end{itemize}
\end{proposition}
\begin{proof} For the sake of brevity we shall write $V = \varinjlim V_i$ and $W = \varprojlim W_i$.

%(1) Let us show first that $w \longmapsto B(\cdot,w)$ is injective. Assume that $B(\cdot,w)$ is the zero homomorphism, so that $B(v,w) = 0$ for every $v \in V$. By the definition of a direct limit this is equivalent to saying that $B(\alpha_i(v_i),w) = 0$ for every $v_i \in V_i$ and every $i$. Now, using Equation \eqref{eq:Bintb} we may rewrite this as $B_i(v_i,\beta_i(w)) = 0$ for every $v_i \in V_i$ and every $i$. By assumption, the canonical map $v_i \longmapsto B(v_i,\cdot)$ is injective and so $\beta_i(w) = 0$. Since this holds true for every $i$, it follows that $w = 0$.

(1) We only prove that $w \longmapsto B(\cdot,w)$ is surjective, leave injectivity to the reader. Let $f \colon V \to \mathbb{K}$ be a homomorphism. For each index $i$ define $f_i := f \circ \alpha_i$, which is a homomorphism from $V_i$ to $\mathbb{K}$. Since the canonical map $w_i \longmapsto B_i(\cdot,w_i)$ is an isomorphism by assumption, there exists a unique $w_i \in W_i$ such that $f_i(v_i) = B_i(v_i,w_i)$ for each $v_i \in V_i$. Now, for any other index $j \geq i$ we also have a homomorphism $f_j$ and an element $w_j \in W_j$ constructed in exactly the same manner and such that $f_j(v_j) = B_j(v_j,w_j)$ for every $v_j \in V_j$. We claim that $\beta_{ij}(w_j) = w_i$. Indeed, for any $v_i \in V_i$ we have \[B_i(v_i,\beta_{ij}(w_j)) = B_j(\alpha_{ij}(v_i),w_j) = f_j(\alpha_{ij}(v_i)) = f_i(v_i)\] %where in the first equality we have used Equation \eqref{eq:intertwine}, in the second one the defining property of $w_j$, and in the third one the fact that $f_j = f_i \circ \alpha_{ij}$ because $\alpha_i = \alpha_j \circ \alpha_{ij}$.
Thus $\beta_{ij}(w_j)$ also satisfies the defining property of $w_i$ and, by the uniqueness of $w_i$, we conclude that $w_i = \beta_{ij}(w_j)$ as claimed. This implies that $(w_i)_{i \in F}$ is an element of the inverse limit $W = \varprojlim W_i$ or, in other words, there exists $w \in W$ such that $\beta_i(w) = w_i$ for every $i \in F$. But then for any element $v \in V$ we may choose and index $i$ and an element $v_i \in V_i$ such that $\alpha_i(v_i) = v$, so that \[f(v) = f_i(v_i) = B_i(v_i,w_i) = B_i(v_i,\beta_i(w)) = B(v,w)\] where in the last step we have made use of Equation \eqref{eq:Binta}.

(2) This follows from part (1). For if $v \neq 0$, there is a linear map $f \colon V \to \mathbb K$ such that $f(v) \neq 0$ and by (1) $f = B( \cdot, w)$ for some $w \in W$. Since $B(v,w) = f(v) \neq 0$, $B(v, \cdot)$ is not identically zero.

\end{proof}

\begin{proof}[Proof of Theorem \ref{teo:nondeg}]
Apply Proposition \ref{prop:nondegabstract} to the bilinear pairings determined by the integral at scales defined by finite open coverings of $X$. This can be done because Proposition \ref{prop:nondeg_U} ensures that at each level $\mathcal{U}$ the integral $\int : H^q_{\mathcal{U}}(X) \times H_q^{\mathcal{U}}(X) \to \mathbb{K}$ is nondegenerate; that is, both canonical maps are isomorphisms.
The theorem follows from the fact that finite open coverings are cofinal among all open coverings of $X$.
%The assumption that $X$ is compact implies that finite coverings are cofinal in the family of all open coverings of $X$. Thus we may apply Proposition \ref{prop:nondegabstract} letting the indices $i$ run in the collection of finite coverings of $X$.

%The theorem is then a direct consequence of the preceding proposition.
\end{proof}

\section{Application: generalization of a theorem of Manning on entropy}\label{sec:manning}

In the following we assume $X$ is a compact space and $f \colon X \to X$ a continuous map. Let us recall first the definition of entropy of a map. This notion roughly computes the base of the exponential growth rate of the number of pieces of orbits that are distinguishable, in the sense that they lie at least $\epsilon$ apart. Since our discussion is purely topological, we use the definition of topological entropy by Adler, Konheim and McAndrew \cite{adler}, where proximity is interpreted in terms of the elements of an open covering. Given an open covering $\mathcal U$ of $X$ we let $s(\mathcal U)$ denote the minimum number of elements of $\mathcal U$ that covers $X$, that is, the smallest size of a subcover of $\mathcal{U}$ (which is finite by compactness). If we set $\mathcal U^n = \mathcal U \vee f^{-1}\mathcal U \vee \cdots \vee f^{-n}\mathcal U$, the limit
\[
\lim_{n \to +\infty} \frac{s(\mathcal U^n)}{n}
\]
exists and is denoted by $h(f, \mathcal U)$. It satisfies the relation: $\mathcal V$ refines $\mathcal U \to h(f, \mathcal V) \ge h(f, \mathcal U)$.

\begin{definition}
The topological entropy of $f$ is $h(f) = \sup h(f, \mathcal U)$, where the supremum ranges among all open coverings $\mathcal U$ of $X$.
\end{definition}

Of course, this definition of topological entropy is equivalent in a metric space to the definitions that use $(n, \epsilon)$-separated or spanning sets which are more common in dynamical systems \cite{katokhassel1}.

In \cite{manning} Manning proved the following theorem: if $f \colon X \to X$ is a continuous map of a compact manifold $X$, then its entropy $h(f)$ is bounded below by the logarithm of the spectral radius of $f_* : H_1(X;\mathbb{C}) \to H_1(X;\mathbb{C})$. That is, $h(f) \geq \log |\lambda|$ for any eigenvalue $\lambda$ of $f_*$. In his paper Manning actually showed that his arguments work in more general compact metric spaces. Theorem 2 in \cite{manning} states that the inequality holds as long as ``two local niceness properties'' are satisfied: the first one is slightly weaker than local path-connectedness and the second property asks for any small loop to be homotopically trivial. As pointed out in the Introduction, Manning believed that \v{C}ech cohomology was more appropriate to relate entropy to eigenvalues. This is the language employed in the generalizations of the result of Manning to arbitrary compact spaces $X$ presented below.

The original proof of Manning roughly kept track of the length of the iterates $f^n(\gamma)$ of a path $\gamma$. In our work, while there are several technical issues that arise from the possible complicated local topology of $X$, the philosophy of the proof is somehow similar. The integral plays a capital role in our argument as the measure of length is replaced by the integral of some cohomology class.

Let us state the theorems.

\begin{theorem} \label{teo:manning1} Let $X$ be compact and locally connected and $f \colon X \to X$ be continuous. Assume that $f^* \colon \check{H}^1(X;\mathbb{C}) \to \check{H}^1(X;\mathbb{C})$ has an eigenvalue $\lambda \in \mathbb{C}$ with modulus $|\lambda| > 1$. Then the topological entropy of $f$ satisfies $h(f) \geq \log |\lambda|$.
\end{theorem}

Notice that this generalizes the classical Manning's inequality, since for a manifold one has $\check{H}^1(X;\mathbb{C}) = H^1(X;\mathbb{C})$ and for a compact manifold the latter is finite dimensional and isomorphic to the dual of $H_1(X;\mathbb{C})$.  

When $X$ is not locally connected the lower bound we obtain is smaller than the expected $\log |\lambda|$ but is still positive, so it ensures that $h(f) > 0$:

\begin{theorem} \label{teo:nmanning} Let $X$ be a compact space and $f \colon X \to X$ a continuous map. Assume that $f^* \colon \check{H}^1(X;\mathbb{C}) \to \check{H}^1(X;\mathbb{C})$ has an eigenvalue $\lambda \in \mathbb{C}$ with modulus $|\lambda| > 1$. Then $h(f) \ge (\mathrm{log}|\lambda|)/d$, where $d \in \mathbb Z^+$ is the degree of the algebraic number $\lambda$. In particular, $h(f) > 0$.
\end{theorem}

Since the degree of rational numbers is 1, the theorem yields the standard bound of $\mathrm{log}|\lambda|$ for eigenvalues in $\mathbb Q$.

%It is possible to give an explicit lower bound on $h(f)$; namely \[h(f) \geq \frac{\log |\lambda|}{d^3+2d+1}\] where $d$ is a certain natural number that depends only on $\lambda$ and whose meaning we shall explain later. It is clear from the structure of the formula that this lower bound is always worse than the usual one $\log |\lambda|$, and this is the penalty for the increase in generality. However, the bound is still sufficient to ensure a positive entropy when $|\lambda| > 1$, which is often enough in applications. 

The proofs of the two theorems have a common initial part given in Subsection \ref{subsec:common} below. After that their proofs diverge. That of Theorem \ref{teo:manning1} (the locally connected case) is quite straightforward using the machinery already developed and is given in Subsection \ref{subsec:final}. The proof of Theorem \ref{teo:nmanning} is more involved and so we devote a full section (Section \ref{sec:nmanning}) to it. Before embarking on this, however, we explain why the theorems involve eigenvalues in cohomology instead of homology.

\subsection{Why are the statements formulated in cohomological terms?}
Another aspect in which Theorems \ref{teo:manning1} and \ref{teo:nmanning} differ from the original result of Manning is that they assume $\lambda$ to be an eigenvalue in (\v{C}ech) cohomology rather than homology. When $\check{H}^1(X;\mathbb{C})$ or $\check{H}_1(X;\mathbb{C})$ are finite dimensional these two spaces are dual to each other (see Theorem \ref{teo:nondeg}) and so one can equivalently assume $\lambda$ to be an eigenvalue in homology. However, as proved earlier in general \v{C}ech homology is isomorphic to the dual of \v{C}ech cohomology, so there might exist eigenvalues in homology which are not present in cohomology and violate Manning's inequality. We now describe an example of compact, locally connected metric space $X$ where this phenomenon occurs.

The space $X$ is shown in Figure \ref{fig:counter1}. It consists of a biinfinite sequence of circumferences $\{C_i\}_{i \in \mathbb{Z}}$ labeled from left to right (with $C_0$ being the biggest one in the middle) and two limiting points $L$ and $R$ at both ends of the sequence.

\begin{figure}[ht]\label{fig:counter1}
\centering
\includegraphics[scale = 0.75]{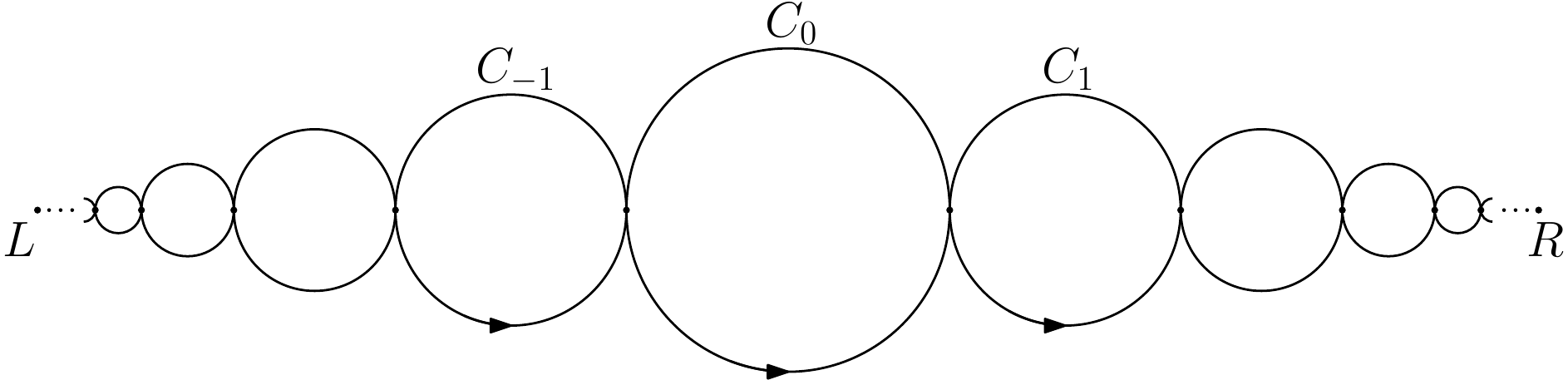}
\caption{}
\label{fig:counter1}
\end{figure}

%In order to describe the map $f : X \to X$ it is convenient to write each $C_i$ as the union of two oriented arcs $a_i$ and $b_i$; these being the upper and lower half of $C_i$ oriented as shown in the figure.
The map $f \colon X \to X$ is defined as follows:
\begin{itemize}
	\item[(i)] It fixes the endpoints $L$ and $R$. Other than that, $f$ sends each point $C_{i+1} \cap C_i$ to $C_i \cap C_{i-1}$.
	\item[(ii)] $f$ sends each oriented (as in the figure) $C_i$ onto a curve that goes around $C_{i-1}$ twice.
%	\item[(ii)] $f$ sends each arc $a_i$ onto the circumference $C_{i-1}$ in such a way that $f(a_i)$ travels counterclockwise around $C_{i-1}$ one and a half times. More precisely, $f$ sends $a_i$ to the concatenation $a_{i-1}b_{i-1}a_{i-1}$.
%	\item[(iii)] $f$ sends each arc $b_i$ homeomorphically onto $b_{i-1}$ (preserving their orientation).
\end{itemize}

Notice that conditions (i) and (ii) are compatible but force $f$ to behave differently in the upper and lower arcs.
%Notice that conditions (ii) and (iii) are indeed compatible with (i). Their joint effect is that $f$ sends each oriented $C_i$ onto a curve that goes around $C_{i-1}$ twice.
Now if $\gamma \in \check{H}_1(X;\mathbb{C})$ is the ``sum'' $\sum_i C_i$, we see that $f_*(\gamma) = 2 \gamma$, so that $\lambda = 2$ is an eigenvalue of $f_*$.

The computation of the entropy is straightforward as it can be circumscribed to the non--wandering set, which in this case is just $\{L,R\}$, an attractor-repeller decomposition of $X$. (For a discussion and proof of this fact see \cite[p. 130ff.]{katokhassel1}). This implies that $h(f) = h(f|_{\{L,R\}}) = 0$, so Manning's inequality fails to be true. In cohomology one sees easily that $f^*$ has no eigenvalues and so Theorem \ref{teo:manning1} holds vacuously.

Another way of manufacturing counterexamples to the homological version of the theorem, albeit non locally connected ones, is to exploit the fact that Manning's inequality may be strongly violated in dimension zero. As an illustration let us take \cite[Corollary 3]{ergodic1}: if $Z$ is a compact, metric, totally disconnected space and $g \colon Z \to Z$ is a continuous map having a nonzero topological entropy, then every $\lambda \in \mathbb{C}$ with $|\lambda| \neq 0,1$ is an eigenvalue of $g_* \colon \check{H}_0(Z;\mathbb{C}) \to \check{H}_0(Z;\mathbb{C})$. In particular if $g$ has a finite entropy then the inequality $h(g) \geq \log |\lambda|$ for every eigenvalue $\lambda$ is manifestly false. This is the case, for example, of Smale's horseshoe with its usual dynamics which has an entropy $h(g) = \log 2$ (see \cite[2.5c and p. 121]{katokhassel1}).

Let $SZ$ be the suspension of $Z$. From the fact that the Mayer-Vietoris sequence for \v{C}ech homology with coefficients in a field is exact \cite{Eilenberg_and_Steenrod_1952_A}, it can be easily deduced that $\check{H}_1(SZ;\mathbb{C})$ is isomorphic to $\check{H}_0(Z;\mathbb{C})$, where reduced $0$--dimensional homology is intended. Then, we can plug the dynamics of $g$ in the base level $Z \times \{0\}$ of $SZ$ and extend it to a dynamics on $f \colon SZ \to SZ$ that fixes the poles $N, S$, preserves the upper and lower cones and makes $(Z \times \{0\}, \{N, S\})$ an attractor-repeller decomposition. As above, the entropy of $f$ is readily seen to be equal to $h(g)$. However, the maps $f_*$ and $g_*$ induced in $\check{H}_1(SZ;\mathbb{C})$ and $\check{H}_0(Z;\mathbb{C})$ are conjugate by $\Delta$, the connecting homomorphism of the Mayer-Vietoris sequence associated to the decomposition of $SZ$ in upper and lower cones.  In particular both maps have the same eigenvalues, which shows that Manning's inequality fails for $f$ since it fails (by assumption) in dimension zero for $g$.

\subsection{Proof of Theorems \ref{teo:manning1} and \ref{teo:nmanning}} \label{subsec:common}

The first steps are common and serve as an outline of both proofs. The argument starts with a nonzero eigenvector $z$ in $\check{H}^*(X; \mathbb C)$ of eigenvalue $\lambda$ that is arbitrary in the proof of Theorem \ref{teo:manning1} but needs to be carefully selected (in a manner to be described in the next section) for the proof of Theorem \ref{teo:nmanning}.

\textbf{Step 1}. 
Let $z \in \check{H}^1(X; \mathbb C)$ be an eigenvector of eigenvalue $\lambda$, i.e. a solution of $f^*z = \lambda z$ (*), and $\mathcal V$ an open covering of $X$ such that $z = \pi_{\mathcal V}(z_{\mathcal V})$ for some $z_{\mathcal V} \in H^1_{\mathcal V}(X; \mathbb C)$. Since the cohomology classes $\lambda z_{\mathcal V}$ and $f^*z_{\mathcal V} \in H^*_{f^{-1}\mathcal V}(X; \mathbb C)$ define the same element in the direct limit $\check{H}^1(X; \mathbb C)$, there exists some open covering $\mathcal U$, finer than $\mathcal V \vee f^{-1}\mathcal V$, for which the projections of $\lambda z_{\mathcal V}$ and $f^* z_{\mathcal V}$ to $H^1_{\mathcal U}(X; \mathbb C)$ are equal. In other words, $z_{\mathcal V}$ (or, formally, $\pi_{\mathcal V \mathcal U}(z_{\mathcal V})$) is a solution of (*) at scale $\mathcal U$. Note that $\mathcal U$ can be chosen arbitrarily fine and will henceforth be fixed.

\textbf{Step 2}. 
Fix a natural number $n$ and denote $\mathcal V_n$ a finite subcover of $\mathcal U \vee f^{-1} \mathcal U \vee \cdots \vee f^{-n} \mathcal U$. For any given $\mathcal V_n$--small 1--cycle $s_n$, the cycles $s_n, f_{\sharp}s_n, \ldots, f_{\sharp}^ns_n$ are $\mathcal U$--small and represent homology classes $\gamma_n, f_*\gamma_n, \ldots, f^n_*\gamma_n \in H_1^{\mathcal U}(X)$, respectively. For every $0 \le j < n$

\[
\int_{f^{j+1}_*\gamma_n} z = \int_{f^j\gamma_n} f^*z = \lambda \int_{f^j\gamma_n} z,
\]
where the middle term shall be interpreted as an integral at scale $f^{-1} \mathcal U$, whereas the first and last belong to scale $\mathcal U$ and we use the properties of the integral formulated in Remark \ref{rmk:scales} and Lemma \ref{lem:integral_map}. By induction,
\begin{align}\label{eq:induccion}
\int_{f^n_*\gamma_n} z = \lambda^n \int_{\gamma_n} z
\end{align}

\textbf{Step 3}. By Corollary \ref{prop:imagenfinita}, there exists a representative $\eta$ of $z$ whose image is finite. Then,
\begin{align}\label{eq:bound}
\left| \int_{f^n_*\gamma_n} z \right| = \left| \int_{[f_{\sharp}^n s_n]} [\eta] \right| \le ||f_{\sharp}^n s_n||_1 ||\eta||_{\infty} \le ||s_n||_1 ||\eta||_{\infty}.
\end{align}
where the first inequality is a consequence of the definition of the integral as a sum of eva\-luations over simplices and the second inequality accounts for the possible alteration of the norm of the chains due to cancellation of simplices after applying $f_{\sharp}^n$ to the chain $s_n$. 

\medskip

At this point the arguments needed to address Theorems \ref{teo:manning1} and \ref{teo:nmanning} diverge.
The idea is to bound $||s_n||_1$ from below in terms of $|\mathcal V_n|$ in (\ref{eq:bound}) and control the integral of $z$ along $\gamma_n$ in (\ref{eq:induccion}) at the same time. This is easy in the locally connected case using the refinements constructed in page \pageref{pg:refine} because the value of the integral does not depend on $n$. However, in the general case there is no easy and coherent choice of $\gamma_n$ and $s_n$ and the argument is more delicate. The idea is to use some arithmetical properties of $z$ to deduce a lower bound for the absolute value of the integral in terms of $|\mathcal V_n|$.

%The missing steps are: to find suitable cycles $s_n$ with controlled norm and to bound the integral on the right hand side term of (\ref{eq:induccion}) and combine it with (\ref{eq:bound}) to conclude that $|\mathcal V_n|$ grows at least exponentially. Next, we see that in the locally connected case the mentioned bound can be made not to depend on $n$.

\subsection{Final part of the proof of Theorem \ref{teo:manning1}} \label{subsec:final}
Using the local connectedness of $X$, after passing to a refinement we can assume that the open covering $\mathcal U$ is composed of connected sets. Then, we can use the arguments in page \pageref{pg:refine} to choose wisely the cycles $s_n$. Take a $\mathcal U$--small 1-cycle $s_0$ that defines a class $\gamma = [s_0] \in H_1^{\mathcal U}(X)$ such that $\int_{\gamma}z \neq 0$. This exists by Proposition \ref{prop:nondeg_U}. For every $n$, let $s_n$ be the $\mathcal V_n$--small refinement of $s_0$ produced by Lemma \ref{lem:count_refine}, which satisfies $||s_n||_1 \le ||s_0||_1|\mathcal V_n|$. Since $s_n$ and $s_0$ are homologous as $\mathcal U$--small cycles,
\[
\int_{\gamma_n} z = \int_{\gamma}z
\]
and then Equation (\ref{eq:bound}) yields
\[
|\mathcal V_n| \ge C \, |\lambda|^n \quad \text{where} \quad C = \frac{1}{||\eta||_{\infty} ||s_0||_1}\left|\int_{\gamma}z\right|
\]
It is then clear that $\lim_{n \to +\infty} \frac{1}{n} \log|\mathcal V_n| \ge \log|\lambda|$ and, since $\mathcal V_n$ was a subcover of $\mathcal U \vee f^{-1} \mathcal U \vee \cdots \vee f^{-n} \mathcal U$ it follows that $h(f, \mathcal U) \ge \log |\lambda|$. Thus, we conclude that $h(f) \ge \log |\lambda|$. \hfill \qed

\section{Proof of Theorem \ref{teo:nmanning}} \label{sec:nmanning}

The case in which $X$ is not locally connected is more difficult as it is not clear whether there is a choice of cycles $s_n$ that allows to control the right hand side of (\ref{eq:induccion}) and the norm of $s_n$ at the same time. The argument we provide is based on a careful inspection of $z$ and $s_n$. Although the eigenvector $z$ is a cohomology class in $H^1_{\mathcal U}(X; \mathbb C)$, we will prove that we may assume $z$ to be a linear combination of rational cohomology classes with algebraic coefficients. Then, we benefit from the nice arithmetic properties of algebraic numbers to obtain a Diophantine lemma that bounds from below the absolute value of the integral of $z$ along $\gamma_n = [s_n]$ (provided it is non-zero) in terms of the norm of $s_n$. Finally, we have to choose $s_n$ in a way that its norm is controlled by the size of $\mathcal V_n$ and that guarantees that the value of the integral of $z$ does not vanish.

The preliminary work for the choice of $s_n$ has already been done in Section \ref{sec:1dim}, as it is enough to pick a suitable simple elementary cycle. However, we still have to go through some technical lemmas to justify the arithmetic properties of $z$. %The following subsection shows that we can assume that $z$ is a linear combination of rational cohomology classes with algebraic coefficients. Afterwards, we state the Diophantine lemma that is key to the proof.

%We need several preparatory lemmas that are fairly distinct in nature. The first one is essentially geometric \textcolor{red}{WHAT??} and somewhat similar to one of the steps in Manning's proof \cite{manning}. The second one is actually number theoretic: it is a Diophantine result which makes use of (a generalization of) the result of Liouville that roughly states that irrational algebraic numbers are badly approximated by rational numbers.

\subsection{Remarks about coefficients. Selection of eigenvector}

From now on we are going to make use simultaneously of coefficients in $\mathbb{Z}$, $\mathbb{Q}$ and $\mathbb{C}$, so these will always be reflected in the notation.
Observe first that a homology (or cohomology, for which the remark also applies verbatim) class $\gamma$ with coefficients in $G$ can be also regarded as a class with coefficients in any larger group $G'$, for example the case of $\mathbb Z$ or $\mathbb Q$ and $\mathbb C$. In view of the definition of the integral this subtlety does not affect the computation as a representative $c$ of $\gamma$ with coefficients in $G$ is also a representative for the class with coefficients in $G'$.

 We also need a brief discussion on the relationship between $H^1_{\mathcal{U}}(X;\mathbb{Q})$ and $H^1_{\mathcal{U}}(X;\mathbb{C})$. By the universal coefficient theorem \cite{spanier1} (and using that $\mathcal{U}$ is finite and so all modules are of finite type), $H^q_{\mathcal{U}}(X;\mathbb{Q}) \otimes \mathbb{C} \cong H^q_{\mathcal{U}}(X;\mathbb{C})$ as $\mathbb Q$--vector spaces, where the isomorphism is given on the generators of the tensor product by $[\xi] \otimes \lambda \longmapsto \lambda [\xi]$. %As in the previous isomorphism, throughout this discussion the tensor product is taken by default over $\mathbb Z$ or, equivalently, over $\mathbb Q$.

Suppose $f \colon X \to X$ is a continuous map. Let us denote by $f^*_{\mathbb{Q}}$ and $f^*_{\mathbb{C}}$ the endomorphisms induced by $f$ in $\check{H}^1(X;\mathbb{Q})$ and $\check{H}^1(X;\mathbb{C})$, respectively, and denote by $P_{\lambda}(t) \in \mathbb Q[t]$ the minimal polynomial of $\lambda$ over $\mathbb Q$ in the sense of field extensions.

\begin{lemma} \label{lem:reduceQ} If $f^*_{\mathbb{C}}$ has an eigenvalue $\lambda \in \mathbb{C}$, then there exists a nonzero $w \in \check{H}^1(X;\mathbb{Q})$ whose minimal polynomial for $f^*_{\mathbb{Q}}$ is well defined and equals $P_{\lambda}$.
\end{lemma}

\begin{proof}
Since the lemma is purely algebraic we shall give a proof that works in a more general setting. Let $E$ be an abstract $\mathbb Q$-vector space and $g \colon E \to E$ an endomorphism of $E$. The $\mathbb C$-vector space $E \otimes \mathbb C$ is generated by the elements of the form $u \otimes 1$, $u \in E$ and the map $g$ induces an endomorphism $\hat g \colon E \otimes \mathbb C \to E \otimes \mathbb C$ which acts on the generators as $\hat g(u \otimes 1) = g(u) \otimes 1$. We view $E \otimes \mathbb{C}$ as a $\mathbb C$--vector space with the complex multiplication absorbed by the second factor: $\nu \cdot u \otimes \mu = u \otimes \nu\mu$. In our setting $g$ corresponds to $f^*_{\mathbb Q}$ and $\hat g$ to the conjugation of $f^*_{\mathbb C}$ by the isomorphism of the universal coefficient theorem.

Any subspace $F \subset E$ spanned by a finite collection $\{u_j\}$ of linearly independent vectors has an associated ($\mathbb C$--linear) subspace $\widehat F$ of $E \otimes \mathbb C$ spanned by $\{u_j \otimes 1\}$. Incidentally, note that $\{u_j \otimes 1\}$ are linearly independent in $E \otimes \mathbb C$ because they are clearly independent as elements of the complexification of $E$, $E^{\mathbb C} = E \otimes_{\mathbb R} \mathbb C$. One can check that:
\centerline{
$\widehat{F_1 \cap F_2} = \widehat F_1 \cap \widehat F_2$ \qquad and \qquad $\widehat{g(F)} = \hat g(\widehat F)$.
}

Assume that $\hat{g}$ has an eigenvector $v \in E \otimes \mathbb{C}$ of eigenvalue $\lambda$. Taking $\{u_j\}$ to be a basis of $E$ we have that $\{u_j \otimes 1\}$ is a basis of $E \otimes \mathbb{C}$, and so $v$ is a linear combination of finitely many of them. Thus there exists an $F \subset E$ spanned by finitely many of the $u_j$ such that $v \in \hat{F}$. Among all such finite dimensional $\mathbb Q$--linear subspaces $F$ of $E$, choose one with the smallest dimension. Then we must have $g(F) = F$ since, otherwise, we could replace $F$ with $F \cap g(F)$ because $v \in \widehat{F} \cap \hat{g}(\widehat{F}) = \widehat{ F \cap g(F)}$.

The minimal polynomials of $\hat{g}$ in $\widehat F$ and $g$ in $F$ coincide. Denote the latter by $P_F$. Then, $P_F(\lambda) = 0$, so $P_F = Q \cdot P_{\lambda}$ for some $Q(t) \in \mathbb Q[t]$. The conclusion follows from the fact that any nonzero vector in $\mathrm{Im}(Q(g)_{|F})$ has minimal polynomial equal to $P_{\lambda}$.
\end{proof}

\begin{remark} In the sequel we will make no notational distinction between $f^*_{\mathbb{Q}}$ and $f^*_{\mathbb{C}}$ and $d$ will denote the degree of $\lambda$ over $\mathbb Q$, $d = \mathrm{deg}(\lambda) = \mathrm{deg}(P_{\lambda})$.
\end{remark}

\begin{lemma}\label{lem:algebraiccombination}
If $f^*$ has an eigenvalue $\lambda \in \mathbb C$, there exists an eigenvector $z \in \check{H}^1(X; \mathbb C)$ for $\lambda$ of the form $z = \sum_{j=0}^{d-1} \mu_j w_j$, where $\mu_j \in \mathbb C$ are algebraic numbers and $w_j \in \check{H}^1(X; \mathbb Q)$.
\end{lemma}

\begin{proof}
By Lemma \ref{lem:reduceQ}, there exists $w \neq 0 \in \check{H}^1(X;\mathbb{Q})$ whose minimal polynomial is $P_{\lambda}(t)$. Let $\lambda = \lambda_1, \ldots, \lambda_d \in \mathbb{C}$ be its complex roots. Then, $P_{\lambda}(t) = (t-\lambda)\cdot (t - \lambda_2) \ldots (t-\lambda_d)$.

Observe that $\{w,f^*w, \ldots, (f^*)^{d-1}w\}$ are linearly independent in $\check{H}^1(X;\mathbb{Q})$ and therefore also in $\check{H}^1(X;\mathbb{C})$. Define an element $z \in \check{H}^1(X;\mathbb{C})$ by \[z := (f^* - \lambda_2) \ldots (f^* - \lambda_d) w.\] (If $\mathrm{deg}(P_{\lambda}) = 1$ then $z = w$). Evidently $(f^* - \lambda) z = P_{\lambda}(f^*)w = 0$, and so $f^* z = \lambda z$. Also, $z$ is nonzero. To check this expand its definition to get
\begin{align*} z & = (f^*)^{d-1} w - (\lambda_2 + \ldots + \lambda_d) (f^*)^{d-2} w + \ldots + (-1)^{d-1} \lambda_2 \ldots \lambda_d w \\
& = \mu_{0} (f^*)^{d-1}w + \mu_{1} (f^*)^{d-2}w + \ldots + \mu_{d-1} w
\end{align*}
which is a linear combination of $\{w, f^*w, \ldots, (f^*)^{d-1}w\}$ with complex coefficients $\mu_j$ at least one of which is nonzero ($\mu_0 = 1$). In particular, $z \neq 0$ and $z$ is an eigenvector of eigenvalue $\lambda$.

Finally, observe that the $\mu_j$ are all algebraic. Indeed, all the $\mu_j$ (which are sums and products of the $\lambda_j$) are algebraic numbers since the set of algebraic numbers is a field (see for instance \cite[Corollary 2.6, p. 232]{lang1}).
\end{proof}

\begin{remark}\label{rmk:eigenvectorU}
In a very similar way to Step 1 in Subsection \ref{subsec:common}, we can assume that the description of $z$ in the lemma is valid at scale $\mathcal U$. Indeed, we can find an open covering $\mathcal V$ of $X$ and classes $z' \in H^1_{\mathcal U}(X; \mathbb C), w_j' \in H^1_{\mathcal U}(X; \mathbb Q)$ such that $\pi_{\mathcal V}(z') = z, \pi_{\mathcal V}(w_j') = w_j$. Then, for some $\mathcal U$ finer than $\mathcal V$, $\pi_{\mathcal V \mathcal U}(\sum \mu_j w'_j) = \pi_{\mathcal V \mathcal U}(z')$. So, with an already familiar abuse of notation 

\centerline{$z = \sum \mu_j w_j$ in $H^1_{\mathcal U}(X; \mathbb C)$,}

\noindent where $w_j \in H^1_{\mathcal U}(X; \mathbb Q)$. As in Step 1, by refining $\mathcal U$ even further we can also guarantee that the eigenvector equation $f^*z = \lambda z$ holds in $H^1_{\mathcal U}(X; \mathbb C)$.
\end{remark}

\subsection{A Diophantine approximation lemma}
Let us start with a theorem due to Schmidt that generalizes several celebrated classical results that roughly state that irrational algebraic numbers are badly approximated by rational numbers. The precise result is Theorem 2 from \cite{schmidt} (see also \cite{waldschmidt1}):

\begin{theorem}
Let $\nu_1, \ldots, \nu_m$ be real algebraic numbers such that $\{1, \nu_1, \ldots, \nu_m\}$ are linearly independent over $\mathbb Q$. For every $\epsilon > 0$ there are only finitely many $m$--tuples of nonzero integers $a_1, \ldots a_m$ with
\[
\mathrm{dist}(a_0 +  \nu_1 a_1 + \ldots + \nu_m a_m, \mathbb Z) \ge \frac{1}{|a_1 \cdots a_m|^{1 + \epsilon}}
\]
\end{theorem}
The inequality immediately implies
\begin{equation}\label{eq:schmidt}
\left|a_0 + \nu_1 a_1 + \ldots + \nu_m a_m\right| \ge A \cdot(\max|a_i|)^{-(m+\epsilon)}
\end{equation}
for all $(m+1)$--tuples of integers $a_j$, with $a_j \neq 0$ if $j > 0$, and some constant $A$ that only depends on $\epsilon$. Now, if we assume that the left hand side of the inequality is not zero, at the expense of replacing $A$ by a larger constant we can suppose that $(a_0, a_1, \ldots, a_m)$ ranges over all $\mathbb Z^{m+1}$ and the numbers $1, \nu_1, \ldots, \nu_m$ are not necessarily rationally independent (use the linear relations to simplify in \eqref{eq:schmidt} until it only depends on a maximal subset of independent $\nu_j$, the factors are absorbed by $A$).
Then, it is clear that the inequality also applies to complex $\nu_j$ (note that rational linear independence of complex numbers is weaker than independence of its real or imaginary parts).

Plug $\mu_0 = 1, \mu_1, \ldots, \mu_{d-1}$ from Lemma \ref{lem:algebraiccombination} in \eqref{eq:schmidt} to deduce that for every $\epsilon > 0$ there exists a constant $A > 0$ such that
\begin{equation}\label{eq:liouville}
\left| \sum_{j = 0 }^{d-1} \mu_j C_j \right| \ge \frac{A}{(\max |C_j|)^{d-1+\epsilon}}
\end{equation}
for every $(C_0, C_1, \ldots, C_d) \in \mathbb Z^{d+1}$ unless the sum on the left vanishes.

In view of the decomposition of $z$ proved in Lemma \ref{lem:algebraiccombination}, the possible nonzero values that the integral of $z$ along an arbitrary homology class $\gamma$ may take can be bounded from below in terms of the number of simplices that compose a $\mathcal U$--small representative of $\gamma$.

\begin{lemma} \label{lem:diophantine} Let $\mathcal{U}$ be a finite open covering of $X$. Suppose $z \in H^1_{\mathcal{U}}(X;\mathbb{C})$ has the form \[z = \sum_{j=0}^{d-1} \mu_j w_j\] where each $\mu_j$ is an algebraic complex number and all the $w_j \in H^1_{\mathcal{U}}(X;\mathbb{Q})$. Then, for every $\epsilon > 0$ there exist a constant $B > 0$ such that, for any $\mathcal{U}$--small cycle $[c] \in H_1^{\mathcal{U}}(X;\mathbb{Z})$, the integral $\int_{[c]} z$ is either zero or its absolute value is bounded below as \[\left| \int_{[c]} z \right| \geq \frac{B}{\|c\|_1^{d-1+\epsilon}}.\]
\end{lemma}

\begin{proof} Each $w_j$ is represented by some cocycle $\xi_j$ (with values in $\mathbb{Q}$) whose coboundary va\-nishes at scale $\mathcal{U}$, and we may assume without loss of generality that ${\rm im}\ \xi_j$ is finite by Corollary \ref{prop:imagenfinita}. Thus, there exists an integer $D > 0$ such that the image of each $D \xi_j$ consists of integer numbers.

Write $c = \sum_i k_i \sigma_i$ with $k_i \in \mathbb{Z}$. By definition the integral $\int_{[c]} z = \sum_j \mu_j \int_{[c]} w_j$ is just the sum

\[
S = \displaystyle\sum_j \mu_j \displaystyle\sum_i k_i \xi_j(\sigma_i) \quad \text{so} \quad DS = \displaystyle\sum \mu_j C_j, \quad \text{where we define} \enskip C_j := D \displaystyle\sum_i k_i \xi_j(\sigma_i) \in \mathbb Z.
\]

We can estimate $C_j$ by  $C_j| \leq D  \|\xi_j\|_{\infty} \sum_i |k_i| = D \|\xi_j\|_{\infty} \|c\|_1,$ so in particular 

\[
\max |C_j| \leq D \|c\|_1 M, \quad \text{where} \enskip M := \max \|\xi_j\|_{\infty}.
\]

Now inequality \eqref{eq:liouville} shows that if $S$ is nonzero, then it is bounded below by \[|S| \geq \frac{1}{|D|}\frac{A}{(\max |C_j|)^{d-1+\epsilon}} \geq \frac{A}{D^{d+\epsilon} M^{d-1+\epsilon} \|c\|_1^{d-1+\epsilon}}.\] It only remains to group everything other than $\|c\|_1^{d-1+\epsilon}$ into a constant $B$. This constant depends on $D$ and $M$ (which are fixed once the choice of the cochains $\xi_j$ is done), $A$, which is determined by the $\mu_j$ and depends on $\epsilon$, and $\epsilon$ itself. The result follows.
\end{proof}

\subsection{Proof of Theorem \ref{teo:nmanning}}
Consider the eigenvector $z$ of eigenvalue $\lambda$ from Lemma \ref{lem:algebraiccombination} and the open covering $\mathcal U$ found in Remark \ref{rmk:eigenvectorU}. Recall that from Steps 1-3 in Subsection \ref{subsec:common} we have
\begin{align}\label{eq:final}
|\lambda^n| \left|\int_{[s_n]} z \right| \le ||s_n||_1 ||\eta||_{\infty}
\end{align}
Since $z$ can be thought of as a cohomology class at scale $\mathcal U$ and $z \neq 0 \in \check{H}^1(X; \mathbb C)$, it defines a non-trivial element of $H^1_{\mathcal V_n}(X; \mathbb C)$ which by the non degeneration of the integral at a specific scale (Proposition \ref{prop:nondeg_U}) has a nonzero integral over some element in $H_1^{\mathcal V_n}(X; \mathbb C)$.  By Proposition \ref{prop:scc} the simple elementary cycles generate $H_1^{\mathcal{V}_n}(X;\mathbb{Z})$ and therefore also $H_1^{\mathcal{V}_n}(X;\mathbb{C})$, so in fact there exists a $\mathcal V_n$--small simple elementary cycle $s_n$ such that $\int_{[s_n]} z \neq 0$.

For the rest of the proof consider $\epsilon > 0$ fixed.
Now, by Lemma \ref{lem:diophantine} applied to $z = \sum_{j = 0}^{d-1} \mu_j w_j$ there exist a constant $B > 0$ such that
\[
\left| \int_{[s_n]} z \right| \ge \frac{B}{||s_n||_1^{d-1+\epsilon}}
\]
Combine this equation with (\ref{eq:final}) and recall that, by definition, a $\mathcal{V}_n$--small simple cycle satisfies $\|s_n\|_1 \leq |\mathcal{V}_n|$. We deduce

\[B |\lambda|^n \leq  \|\eta\|_{\infty} |\mathcal{V}_n|^{d+\epsilon}.\]

Notice that in this inequality none of $\|\eta\|_{\infty}, d, \epsilon$ depends on $n$. A straightforward computation then shows that \[\frac{\log |\lambda|}{d+\epsilon} \leq \lim_{n \rightarrow +\infty} \frac{1}{n} \log |\mathcal{V}_n|,\]
and since $\mathcal V_n$ was an arbitrary finite subcover of $\mathcal U \vee \cdots \vee f^{-n}\mathcal U$ for every $n$ (and $\epsilon > 0$ is arbitrary) this bounds $h(f, \mathcal U)$ from below. It follows that $h(f) \ge (\mathrm{log}|\lambda|)/d > 0$. This finishes the proof of Theorem \ref{teo:nmanning}.

%\begin{remark} %To obtain an explicit lower bound for $h(f)$ one only needs to keep track of the value of $L$. This number first appeared in the proof when Liouville's inequality was used and admits the expression $L = D (2 + \sum_{j=0}^{d-1} h(\mu_j))$ where $h(\mu_j)$ is the height of $\mu_j$ and $D$ is the degree of the field extension $\mathbb{Q}(\mu_0,\ldots,\mu_{d-1}) / \mathbb{Q}$. This extension is certainly contained in $\mathbb{Q}(\lambda_1,\ldots,\lambda_d) / \mathbb{Q}$, which has degree at most $d$ since all the $\lambda_i$ are roots of the polynomial $M(t)$ of degree $d$, so $D \leq d$. Bounding the heights $h(\mu_j)$ also by $d$ yields $L \leq d(2+d^2) = d^3 + 2d$. Plugging this into $L+1$ at the end of the proof of the theorem gives the bound \[h(f) \geq \frac{\log |\lambda|}{d^3+2d+1}.\]

%A priori $d$ may depend on the $w$ and $M$ that came out of Lemma \ref{lem:reduceQ}. However, is not difficult to elaborate on the proof of the lemma to show that $M(t)$ is actually the minimal polynomial of $\lambda$ in the sense of field extensions. This means that $M(\lambda) = 0$ and if $P(t) \in \mathbb{Q}[t]$ is any other nonzero polynomial that has $\lambda$ as a root, then $M | P$. Thus $d$ is in fact the degree of $\lambda$ over $\mathbb{Q}$. %For example, for a rational eigenvalue $\lambda$ we then have $d = 1$ and the bound $h(f) \geq \nicefrac{1}{4} \log |\lambda|$. As another example, one always has $d \leq  \dim \check{H}^1(X;\mathbb{C})$ which yields a numerical bound for $h(f)$ when $\check{H}^1(X;\mathbb{C})$ is finite dimensional.
%\end{remark}

\bibliographystyle{plain}
\bibliography{biblio}

\end{document}